\documentclass{article}
\pdfoutput=1
\usepackage{amssymb,amsfonts,amsmath}
\usepackage{amsthm}
\usepackage{algorithm}
\usepackage[noend]{algorithmic}
\usepackage{graphicx}
\usepackage{longtable}

\newtheorem{theorem}{Theorem}[section]
\newtheorem{lemma}[theorem]{Lemma}
\newtheorem{propn}[theorem]{Proposition}
\newtheorem{corol}[theorem]{Corollary}

\newtheorem{definition}[theorem]{Definition}

\newcommand{\newU}{\ensuremath{\mathcal{U}}}
\newcommand{\newV}{\ensuremath{\mathcal{V}}}

\newcommand{\sgn}{\ensuremath{\,\mathrm{sgn}}}

\newcommand{\finfldq}[1]{\ensuremath{\mathbb{F}_{#1}}}
\newcommand{\polyringq}[1]{\ensuremath{\mathbb{F}_{#1}[t]}}

\newcommand{\genlink}[1]{\ensuremath{GL_2(#1)}}
\newcommand{\matlink}[1]{\ensuremath{Mat_2(#1)}}

\newcommand{\rats}{\ensuremath{\mathbb{Q}}}
\newcommand{\ratsfuncqextn}[2]{\ensuremath{\mathbb{F}_{#1}(t,#2)}}
\newcommand{\ratsfuncq}[1]{\ensuremath{\mathbb{F}_{#1}(t)}}

\newcommand{\norm}[1]{\ensuremath{\mathcal{N}(#1)}}

\newcommand{\Fq}{\ensuremath{\mathbb{F}_q}}

\begin{document}

\title{Tabulation of cubic function fields via polynomial binary cubic forms}

\author{Pieter Rozenhart \\Department of Mathematics and Statistics, University of
  Calgary,\\
  2500 University Drive NW, Calgary, Alberta, Canada, T2N 1N4\\ \texttt{pmrozenh@alumni.ucalgary.ca} \and Michael Jacobson Jr.\ \\   Department of Computer Science, University of
  Calgary,\\
  2500 University Drive NW, Calgary, Alberta, Canada, T2N 1N4\\
   \texttt{jacobs@cpsc.ucalgary.ca} 
\and Renate Scheidler \\Department of Mathematics and Statistics, University of
  Calgary,\\
  2500 University Drive NW, Calgary, Alberta, Canada, T2N 1N4\\
    \texttt{rscheidl@math.ucalgary.ca}}





\maketitle
\begin{abstract}
We present a method for tabulating all cubic function fields over $\ratsfuncq{q}$ whose discriminant $D$ has either odd degree or even degree and the leading coefficient of $-3D$ is a non-square in $\finfldq{q}^*$, up to a given bound $B$ on $\deg(D)$.  Our method is based on a generalization of Belabas' method for tabulating cubic number fields.  The main theoretical ingredient is a generalization of a theorem of Davenport and Heilbronn to cubic function fields, along with a reduction theory for binary cubic forms that provides an efficient way to compute equivalence classes of binary cubic forms.  The algorithm requires $O(B^4 q^B)$ field operations as $B \rightarrow \infty$.  The algorithm, examples and numerical data for $q=5,7,11,13$ are included.  
\end{abstract}



%

\section{Introduction and Motivation}
 
 In 1997, Belabas \cite{Bela1} presented an algorithm for
tabulating all non-isomorphic cubic number fields of discriminant
$D$ with $|D| \leq X$ for any $X > 0$.  In the above context, \emph{tabulation}
means that all non-isomorphic fields with discriminant $|D| \leq X$
are listed or written to a file, by listing the minimal polynomial for each respective field.  The results make use of the
reduction theory for binary cubic forms with integral
coefficients. A theorem of Davenport and Heilbronn \cite{HeDav2}
states that there is a discriminant-preserving bijection between
$\rats$-isomorphism classes of cubic number fields of discriminant
$D$ and a certain explicitly characterizable set $\newU$ of
equivalence classes of primitive irreducible integral binary cubic
forms of the same discriminant $D$. Using this one-to-one
correspondence, one can enumerate all cubic number fields of
discriminant $D$ with $|D| \leq X$ by computing the unique reduced
representative $f(x,y)$ of every equivalence class in $\newU$ of
discriminant $D$ with $|D| \leq X$. The corresponding field is
then obtained by simply adjoining a root of the irreducible cubic
$f(x,1)$ to \rats. Belabas' algorithm is essentially linear in
$X$, and performs quite well in practice. 

In this paper, we give an extension of the above approach to
function fields.  That is, we present a method for tabulating all $\ratsfuncq{q}$-isomorphism classes of
cubic function fields over a fixed finite field $\finfldq{q}$ up to a given
upper bound on the degree of the discriminant, using the theory
of binary cubic forms with coefficients in \polyringq{q}, where
\finfldq{q}\ is a finite field with char$(\finfldq{q}) \neq 2,3$.  The discriminant $D$ must also satisfy certain technical conditions, which we detail below.  

This paper is a substantial expansion and extension of the material in \cite{ANTSPie} and corresponds to Chapters 4 and 6 of the first author's Ph.D thesis, prepared under the supervision of the last two authors.  The present paper includes a more detailed description of the algorithm in the unusual discriminant case and improved bounds on the coefficients of a reduced form compared to those appearing in \cite{ANTSPie}.  Another somewhat complementary approach to tabulating cubic function fields different from the one taken in this paper is described in \cite{cuffqi}, and relies on computations of quadratic ideals to construct all cubic function fields of a given fixed discriminant.  The approach in \cite{cuffqi} is limited to finite fields $\finfldq{q}$ where $q \equiv -1 \pmod{3}$, and unlike our algorithm, considers only fixed discriminants and not all discriminants $D$ with $\deg(D)$ bounded above.

Our main tool is the function field analogue of the
Davenport-Heilbronn theorem \cite{HeDav2} mentioned above, which generalizes to Dedekind domains (see Taniguchi \cite{Taniguchi}).  As in the case of integral forms, we also make use of the association of any binary
cubic form $f$ of discriminant $D$ over \polyringq{q}\ to its
Hessian $H_f$ which is a binary quadratic form over \polyringq{q}\
of discriminant $-3D$. Under certain conditions on the discriminant, this association
can be exploited to develop a reduction theory for binary cubic
forms over \polyringq{q} that is analogous to the reduction theory
for integral binary cubic forms. Suppose that $\deg(D)$ is odd, or that
$\deg(D)$ is even and the leading coefficient of $-3D$ is a
non-square in $\finfldq{q}^*$. We will establish that under these conditions, the
equivalence class of $f$ contains a unique reduced form, i.e.\ a
binary cubic form that satisfies certain normalization conditions
and has a reduced Hessian. Thus, equivalence classes of binary cubic forms can be
efficiently identified via their unique representatives.  The case where $\deg(D)$ is odd is analogous to the case of definite binary quadratic forms, but the other case has no number field analogue. 

Our tabulation method proceeds analogously to the number field
scenario. The function field analogue of the Davenport-Heilbronn
theorem states that there is a discriminant-preserving
bijection between \ratsfuncq{q}-isomorphism classes of cubic
function fields of discriminant $D \in \polyringq{q}$ and a
certain set $\newU$ of primitive irreducible binary cubic forms
over \polyringq{q} of discriminant $D$. Hence, in order to list
all \ratsfuncq{q}-isomorphism classes of cubic function fields up
to an upper bound $B$ on $\deg(D)$, it suffices to
enumerate the unique reduced representatives of all equivalence
classes of binary cubic forms of discriminant $D$ for all $D \in$
\polyringq{q} with $\deg(D) \leq B$. Bounds on the
coefficients of such a reduced form show that there are only
finitely many candidates for a fixed
discriminant. These bounds can then be employed in nested loops over the coefficients to
test whether each form found lies in $\newU$. The coefficient bounds obtained for function fields are
different from those used by Belabas for number fields, due to the
fact that the degree valuation is non-Archimedean.  In fact, we obtain far simpler and more elegant  bounds than those in the number field case. 

This paper is organized as follows.   Section \ref{sec:prel} begins with some background material on algebraic function fields.  The reduction theory for imaginary and unusual binary quadratic forms and binary cubic forms over $\polyringq{q}$ is developed in Sections \ref{BQFUnus} and \ref{sec:cubredn}, respectively.  The derivation of the bounds on the coefficients of a reduced binary cubic form appears in Section \ref{sec:coefbdds}.   The Davenport-Heilbronn Theorem for cubic function fields is presented in Section \ref{sec:DH}.
We detail the tabulation algorithm as
well as numerical results in Section \ref{sec:algnum}.  Finally, we conclude with some open problems and future research directions in Section \ref{sec:Concl}.

\section{Preliminaries}\label{sec:prel}

For a general introduction to algebraic function fields, we refer
the reader to Rosen \cite{Rosen} or Stichtenoth \cite{Sticht}. Let
\finfldq{q}\ be a finite field of characteristic at least~$5$, and
set $\finfldq{q}^* = \finfldq{q} \backslash \{0\}$. Denote by
\polyringq{q}\ and \ratsfuncq{q}\ the ring of polynomials and the
field of rational functions in the variable $t$ over \finfldq{q},
respectively.  For any non-zero $H \in \polyringq{q}$ of degree $n
= \deg(H)$, we let $|H| = q^n = q^{\deg(H)}$, and denote by
$\sgn(H)$ the leading coefficient of $H$. For $H = 0$, we set $|H|
= 0$. This absolute value extends in the obvious way to
\ratsfuncq{q}. Note that in contrast to the absolute value on the
rational numbers, the absolute value on \ratsfuncq{q} is non-Archimedean.

An \emph{algebraic function field} is a finite extension $K$ of $\ratsfuncq{q}$; its degree is the
field extension degree $n = [K : \ratsfuncq{q}]$. It is always possible to write a function
field as $K = \ratsfuncqextn{q}{y}$ where $F(t, y) = 0$ and $F(Y )$ is a monic polynomial of degree
$n$ in $Y$ with coefficients in $\polyringq{q}$ that is irreducible over $\ratsfuncq{q}$.   We assume that $\finfldq{q}$ is the full constant field of $K$, i.e. $F(Y)$ is absolutely irreducible.

A homogeneous polynomial in two variables of degree $2$, with
coefficients in \polyringq{q}, of the form
\[ f(x,y) = Ax^2 + Bxy + Cy^2  \]
is called a \emph{binary quadratic form} over \polyringq{q}. We abbreviate the form
as $f = (A,B,C)$.   Similarly, a homogeneous polynomial in two variables of degree $3$,
with coefficients in \polyringq{q}, of the form
\[f(x,y) = ax^3 + bx^2y + cxy^2 + dy^3  \]
is called a \emph{binary cubic form over \polyringq{q}}.  We
abbreviate the form $f(x,y)$ by $f=(a,b,c,d)$.  The discriminant $D$ of a binary quadratic form $f=(A,B,C)$ is $D(f) = B^2 - 4AC$.  In a similar vein, the discriminant of a binary cubic form $f=(a,b,c,d)$ is $D = D(f) = 18abcd + b^2c^2 - 4ac^3 - 4b^3d - 27a^2d^2$.  We will assume that all forms are primitive, irreducible over \polyringq{q}, and have distinct roots and thus non-zero discriminant.

Let $D$ be a polynomial in $\polyringq{q}$.   
Then $D$ is said to be \emph{imaginary}
if $D$ has odd degree, \emph{unusual} if $D$ has even
degree and $\sgn(D)$ is a non-square in
$\finfldq{q}^*$, and \emph{real} if $D$ has even degree and $\sgn(D)$ is a square in $\finfldq{q}^*$.   
Correspondingly, a binary quadratic form is said to be \emph{imaginary, unusual or real} according to whether its discriminant is imaginary, unusual or real.  These terms have their origins in quadratic number fields.  If $D$ is imaginary (resp.\ real), then the quadratic function field $\ratsfuncqextn{q}{\sqrt{D}}$ shows many similarities to an imaginary (resp.\ real) quadratic number field, such as the splitting of the infinite place of $\ratsfuncq{q}$ and the unit group structure.  For $D$ unusual, there is no number field analogue to the function field $\ratsfuncqextn{q}{\sqrt{D}}$.  The terminology ``unusual" is due to Enge \cite{Enge1}, and we adopt this terminology for quadratic fields and binary quadratic forms in this paper.  

Integral binary forms have a rich history going back to Lagrange and Gauss, and many important applications (see Buchmann and Vollmer \cite{BuchVoll} and Buell \cite{Buell1}).  Their reduction theory was developed for positive definite forms first, as this case is the most straightforward.  Recall (from Buell \cite{Buell1}, Chapters 1 and 2, for example) that an integral binary quadratic form $(A,B,C)$ is \emph{definite} if its discriminant is negative.  In this case, both $A$ and $C$ have the same sign.   One then further specializes to \emph{positive definite} forms; these are forms with negative discriminant and $A > 0$ (and hence $C >0$).  In other words, one considers the element $f = (A,B,C)$ in the associate class of definite forms with $A>0$.  Correspondingly, if $f = (A,B,C)$ is a binary quadratic form with imaginary or unusual discriminant, then we say that $f$ is \emph{positive definite} if $\sgn(A)$ is a square in $\finfldq{q}^*$, and \emph{negative definite} otherwise.

We summarize a few useful notions and results pertaining to binary quadratic and cubic forms over $\polyringq{q}$ below. These results are completely analogous to their well-known counterparts for integral binary quadratic forms, found in Chapter 8 of Cohen \cite{Coh2}. 

The set
\[\{ M : M \mbox{ is a $2\times 2$ matrix with entries in $\polyringq{q}$ and } \det(M) \neq 0 \}\]
is denoted by \matlink{\polyringq{q}}.
As usual, the set 
\[\genlink{\polyringq{q}} = \{ M : M \in \matlink{\polyringq{q}} \mbox{  with } \det(M) \in \finfldq{q}^* \}\]
denotes the \emph{general linear group of degree $2$}. 

Let $f$ be a binary quadratic or cubic form and $M = \left( \begin{smallmatrix} \alpha & \beta \\ \gamma & \delta \end{smallmatrix} \right) \in \matlink{\polyringq{q}}$.  
The \emph{action} of
$M$ on $f$ is given by $f \circ M = f(\alpha x + \beta y, \gamma x + \delta y)$.  Using this action, we give the definition of equivalence, which differs slightly from \cite{ANTSPie} in that the extra multiplication of $f$ by a unit is removed.  This was done to simplify some of the subsequent work.
Two binary forms $f$
and $g$ over $\polyringq{q}$ are said to be \emph{equivalent} if
\[f(\alpha x + \beta y, \gamma x + \delta y) = g(x,y)\]
for some $\alpha, \beta, \gamma,
\delta \in \polyringq{q}$ with $\alpha \delta - \beta \gamma \in
\finfldq{q}^*$, i.e., $g = f \circ M$ for some $M$ in $\genlink{\polyringq{q}}$.  We immediately obtain equations for the coefficients of equivalent binary quadratic forms, namely, if $f = (A,B,C)$ and $g=(A',B',C') = f \circ M$, with $M$ as above, then
\begin{eqnarray}
A' & = & A \alpha^2 + B \alpha \gamma + C \gamma^2, \label{A1form}  \\
B' & = & 2A \alpha \beta + B (\alpha \delta + \beta \gamma) + 2C \gamma \delta, \label{B1form}\\
C' & = & A \beta^2 + B \beta \delta + C \delta^2 .  \label{C1form}
\end{eqnarray}
Analogous formulas for binary cubic forms are also easy to obtain, and are omitted.

We note that it indeed possible for a positive definite unusual form to be equivalent to a
negative definite form.  For instance, if $|A| = |C|$ with $\sgn(A) = 1$ (so $f$ is
positive definite) and $q \equiv 1 \pmod{4}$ (so $-h/4$ is a non-square), then
swapping $A$ and $C$ yields an equivalent negative definite form.

\begin{propn}
Let $M = \left( \begin{smallmatrix}
    \alpha & \beta \\ \gamma & \delta \end{smallmatrix} \right) \in \matlink{\polyringq{q}}$ and $f$ a binary
quadratic or cubic form over $\polyringq{q}$.   Then the following hold:
\begin{enumerate}
\item $D(f \circ M) = (\alpha \delta - \beta \gamma)^{2}\,\cdot D(f)$, if $f$ is a binary quadratic form. 
\item $D(f \circ M) = (\alpha \delta - \beta \gamma)^{6}\,\cdot D(f)$, if $f$ is a binary cubic form.
\item  If $M$ is non-singular, then $f \circ M$ is irreducible over $\polyringq{q}$ if and only if $f$ is irreducible.
\item  If $M$ is non-singular, then $f \circ M$ is primitive if and only if $f$ is primitive. 
\end{enumerate}
\label{FstMajCub}
\end{propn}

By Proposition \ref{FstMajCub}, up to an even power of $\det(M)$, equivalent binary quadratic and cubic forms have the
same discriminant. In addition, primitivity and irreducibility are
preserved by the action just defined, provided the transformation matrix is non-singular. 

We finish this section by introducing the Hessian of a binary cubic form, along with some of its properties which are easily verified by straightforward computation.

\begin{definition}
Let $f = (a,b,c,d)$ be a binary cubic form over $\, \polyringq{q}$. 
\emph{The Hessian} of $f$ and the polynomials $P,Q,R$ are given by 
\[H_f(x,y) = -\frac{1}{4} \det \!\left(\begin{array}{cc}
\frac{\partial^2f}{\partial x \partial x} & 
\frac{\partial^2f}{\partial x \partial y} \\
\frac{\partial^2f}{\partial y \partial x} &
\frac{\partial^2f}{\partial y \partial y}\end{array}  \right) = Px^2 +
Qxy + Ry^2,\]
where $P = b^2 - 3ac$, $Q = bc - 9ad$ and $R = c^2 - 3bd$.
\end{definition}

Note that $H_f$ is a binary quadratic form over \polyringq{q}. 

\begin{propn}
Let $f = (a,b,c,d)$ be a binary cubic form over \polyringq{q}\ with Hessian $H_f =
(P,Q,R)$.  Then the following properties are satisfied.
\begin{enumerate}
\item For any $M \in \matlink{\polyringq{q}}$, we have $H_{f \circ M} = (\det
  M)^2\;\cdot H_f \circ M$.
\item $D(H_f) = -3\,\cdot D(f) $.
\end{enumerate}
\label{Hessprops2}
\end{propn}

\section{Reduction of Binary Quadratic Forms over \polyringq{q}}\label{BQFUnus}

In this section, we give the reduction theory for binary quadratic forms with polynomial coefficients, due to Artin \cite{Artin1}.  A modified version of Artin's material is presented here, as he does not consider binary quadratic forms, but only their roots, which results in a simpler treatment.  Furthermore, some of his presentation is streamlined
in this paper, using more modern notation.

The reduction theory for binary quadratic and cubic forms allows us to single out a unique 
representative in each equivalence class of forms.  The theory also
enables the efficient computation of this representative, and demonstrates that
there are only finitely many such equivalence classes for any given
non-zero discriminant $D$ in the case of binary quadratic forms.  The case where $D$ is a real discriminant of a binary quadratic form is excluded from the paper.

The following conventions will be adopted.  As before, let \finfldq{q}\ be a finite field of characteristic at least $5$.  Fix a primitive root $h$ of $\finfldq{q}^*$.  We predefine the set $S = \{ h^i :  0 \leq i \leq (q-3)/2 \}$, so that $a \in S$ if and only if $-a \notin S$.     
As in Artin \cite{Artin1}, the discriminant $D$ of $f$ is also endowed with the normalization $\sgn(D) = 1$ or $\sgn(D) = h$, where $1$ or $h$ is chosen depending on whether or not $\sgn(D)$ is a square in $\finfldq{q}^*$.  We choose this normalization in order to avoid the possibility of forms being equivalent to each other while possessing different discriminants.  By Proposition \ref{FstMajCub}, the discriminant of a form can only change by a square factor of a finite field element, so normalizing to a single square or non-square sign value accomplishes this task.

We provide the reduction theory for binary quadratic forms over $\polyringq{q}$. We will treat the imaginary and unusual scenarios in parallel where possible, but there are significant differences.  The case of unusual binary quadratic forms differs from that of imaginary forms (see \cite{Pieter2}), as it has no number field analogue.  Another crucial difference is as follows:  the analogous definition of ``reduced" does not lead to a unique representative in each equivalence class in the case that $|A|= |C|$, but instead to $q+1$ equivalent forms, called \emph{partially reduced forms}, defined below.  To achieve uniqueness, a distinguished representative among these $q+1$ equivalent partially reduced forms needs to be identified.  Some of the results and proofs below apply to imaginary forms as well, and these occasions will be explicitly mentioned.  

For the remainder of this section, let $f=(A,B,C)$ be an imaginary or unusual binary quadratic form with discriminant $D$.
Recall that we only consider irreducible forms, hence $A \neq 0$.
Furthermore, our earlier assumption that $\finfldq{q}$ is the full constant field of the function field $\ratsfuncqextn{q}{\sqrt{D}}$ implies that $D \notin \finfldq{q}$.

\begin{definition}
An imaginary or unusual binary quadratic form $f = (A, B, C)$ over $\polyringq{q}$ is \emph{partially 
reduced} if it satisfies the following conditions:
\begin{enumerate}
\item $|B| < |A| \leq |C|$;
\item If $|A| < |C|$, then $\sgn(A) \in \{1,h\}$, and if $|A| = |C|$, then $\sgn(A) = 1$;
\item $B \neq 0$ implies $\sgn(B) \in S$. 
\end{enumerate}
\label{partiallyreduced}
\end{definition}

Note that $|B| < |A| \leq |C|$ if and only if $|B| < |A| \leq |D|^{1/2}$, and that reduced unusual forms are always positive definite.  The exponent in $|D|^{1/2} = q^{\deg(D)/2}$ in the above lemma is a half integer in the case where $f$ is an imaginary binary quadratic form, so property (1) above is in fact equivalent to $|B| < |A| < |C|$.
However, in the case of unusual binary quadratic forms, the exponent in $\sqrt{|D|} = q^{\deg(D)/2}$ is an integer.   Hence, equality $|A|= |C| = |D|^{1/2}$ may in fact occur in this case.

\begin{propn}
Every binary quadratic form is equivalent to a partially reduced binary quadratic form $f=(A,B,C)$ of the same discriminant.
\label{preeqquad}
\end{propn}

\begin{proof}
The procedure to find $f$ satisfying property (1) above is completely analogous to the one for integral binary cubic forms; see Algorithms 5.1-5.3, page 87 of Buchmann and Vollmer \cite{BuchVoll}, or \cite{Pieter2}.  Now, let $f = (A,B,C)$ be a binary quadratic form satisfying condition $(1)$ of Definition \ref{partiallyreduced}.
If $|A| < |C|$ and $\sgn(A) \neq 1$ or $h$, then we replace $f(x,y)$ by $f(\epsilon^{-1}x,\epsilon y)$, where $\epsilon^2 = \sgn(A)$ if $f$ is positive definite and $\epsilon^2 = \sgn(A) h^{-1}$ if $f$ is negative definite.  In other words, replace $f$ with $f \circ M$, where $M = \left( \begin{smallmatrix} \epsilon^{-1} & 0 \\ 0 & \epsilon \end{smallmatrix} \right)$.  This will yield a form satisfying $|B| < |A| \leq |C|$ and $\sgn(A) \in \{1, h\}$.

If $|A| = |C|$, then consider the norm map from $\finfldq{q}(\sqrt{h})$ down to $\finfldq{q}$; this map is always surjective (see for example Theorem 2.28, p.\ 54 of \cite{LidlNied}). Thus, $\sgn(A)^{-1} \in \finfldq{q}^*$ is the norm of some element $\alpha + \epsilon \sqrt{h}$ in $\finfldq{q}(\sqrt{h})$ ($\alpha, \epsilon \in \finfldq{q}$). It is now easy to see from equations (\ref{A1form})--(\ref{C1form}) that the transformation matrix
$\left(
\begin{smallmatrix}
\alpha & h \epsilon/2 \\2 \sgn(A) \epsilon & \sgn(A) \alpha
\end{smallmatrix}
\right)$
of determinant $1$ yields a binary quadratic form $f' = (A',B',C')$ with $\sgn(A') = 1$ and $|B'| < |A'| = |C'|$.

Finally, If $B \neq 0$, then $\sgn(B) \in S$ is achieved by transforming $f$ with $J = \left(
\begin{smallmatrix}
1 & 0 \\ 0 & -1
\end{smallmatrix}
\right)$ if necessary.
\end{proof}
 
If $f = (A,B,C)$ is a partially reduced form with $|A|=|C|$ and $-1$ is a non-square in $\finfldq{q}$, then $-4/h = \sgn(C)^{-1}$ is a square in $\finfldq{q}$, say $-4/h = \epsilon^2$. Then $(A,B,C)$ is equivalent to $(\epsilon^2 C, B, \epsilon^{-2} A)$ via the matrix
$\left(
\begin{smallmatrix}
      0       &      \epsilon^{-1} \\
 \epsilon  &     0     
\end{smallmatrix}
\right)
$.  
Hence, additional normalization conditions will be
needed to obtain a unique representative in each equivalence class in the case when $|A| = |C|$.  
We now characterize when two partially reduced binary quadratic forms are equivalent. 

\begin{theorem}
Let $f=(A,B,C)$ and $f'=(A',B',C')$ be two partially reduced binary quadratic forms of the same discriminant.  Then the following hold:
\begin{enumerate}
\item If $|A| < |C|$ and $|A'| < |C'|$, then $f$ and $f'$ are equivalent if and only if $f=f'$.
\item If $|A| < |C|$ and $|A'| = |C'|$ then $f$ and $f'$ are not equivalent.
\item If $|A| = |C|$ and $|A'| = |C'|$ then $f$ and $f'$ are equivalent with $f' = f \circ M$ if and only if 
\[M =  \left( \begin{array}{cc} \alpha & \beta \\ 4u \beta/h & u \alpha \end{array}   \right)  \]
where $\alpha,\beta \in \finfldq{q}$, $\alpha^2 - (4/h) \beta^2 =1$, $u := \det(M) = \pm1$, and if $B' \neq 0$, then $u$ is determined by the condition $\sgn(B') \in S$.
\end{enumerate}
\label{quadeqthm}
\end{theorem}

\begin{proof}

Suppose that $f' = f \circ M$, with $M= \left(
\begin{smallmatrix}
      \alpha       &      \beta \\
 \gamma  &     \delta
\end{smallmatrix}
\right)$ and $u := \det(M)= \alpha\delta - \beta\gamma \in \finfldq{q}^*$.  Then $ u= \pm 1$ by Proposition \ref{FstMajCub}, since $f,f'$ have the same discriminant.

Since $(A',B',C') = f(\alpha x + \beta y, \gamma x + \delta y)$, equations (\ref{A1form})-(\ref{C1form}) hold.  Since $f = f' \circ M^{-1}$, we also obtain:
\begin{eqnarray}
uA & = & A' \delta^2 - B' \gamma \delta + C' \gamma^2, \label{invrel1} \\
uB & = & -2A' \beta \delta + B'(\alpha \delta + \beta \gamma) - 2C' \alpha \gamma, \label{invrel2}   \\ 
uC & = & A'\beta^2 - B' \alpha \beta + C' \alpha^2.  \label{invrel3}  
\end{eqnarray}

\noindent Forming appropriate products of (\ref{A1form}), (\ref{C1form}), (\ref{invrel1}) and (\ref{invrel3}) yields:
\begin{eqnarray}
4uAA' & = & (2A\alpha + B\gamma)^2 - D\gamma^2, \label{sinfor1} \\
4uCA' & = & (2C\gamma + B\alpha)^2 - D\alpha^2, \label{sinfor2}  \\
4uAC' & = & (2A\beta + B\delta)^2 - D\delta^2,  \label{sinfor3}  \\
4uCC' & = & (2C\delta - B \beta)^2 - D\beta^2. \label{sinfor4}
\end{eqnarray}

\noindent Suppose first that $|A| < |C|$, so $|A| < |D|^{1/2}$.
Then by equation (\ref{sinfor1}), $|A'|\leq |D|^{1/2}$ and
\[|(2A\alpha + B\gamma)^2 - D\gamma^2 | < |D|.\]
Since the leading terms of the expressions $(2A\alpha + B\gamma)^2$ and $D\gamma^2$ on the left hand side of the above inequality cannot cancel, this forces $\gamma = 0$. Thus $\det(M) = \alpha\delta$, and hence $\alpha,\delta \in \finfldq{q}^*$.  
Now (\ref{A1form}) implies $|A| = |A'|$. Then $|D| = |AC| = |A'C'|$ precludes the possibility that $|A'| = |C'|$, which proves part (2) of theorem. Furthermore, (\ref{B1form}), together with $|B|, |B'| < |A|$ forces $\beta = 0$.  Thus $\beta = \gamma = 0$.  Then equations (\ref{A1form})-(\ref{C1form}) yield
$A'  =  A \alpha^2$, 
$B'  =  B (\alpha \delta)$,
$C'  =  C \delta^2$.

It follows that $\det(M) = \alpha \delta = \pm 1$.   Since both $\sgn(A)$ and $\sgn(A')$ are either squares (in which case they are both $1$) or non-squares (in which case they are both $h$), we thus obtain $\alpha^2 = 1$.  Hence $\alpha = \pm 1$.  This yields four possibilities for the matrix $M$:  $\pm I, \pm J$, where $J =  \left( \begin{smallmatrix} 1 & 0 \\ 0 & -1    \end{smallmatrix}\right)$.  If $B = 0 = B'$, then $f = f'$.  If $B \neq 0$, then both $\sgn(B), \sgn(B') \in S$, which means that $\alpha \delta = 1$.  This leaves $M = \pm I$, hence $f = f'$, as desired, completing the proof of part $(1)$.

To prove part (3), suppose now that $|A| = |C| = |A'| = |C'| = |D|^{1/2}$; in particular, $D$ is an unusual discriminant.  First, we deduce that $\alpha,\beta,\gamma,\delta \in \finfldq{q}$ under this assertion.
To see this, note that the absolute value of each of the left-hand sides
of equations (\ref{sinfor1})--(\ref{sinfor4}) above equals $|D|$.    Note also that there cannot be cancellation of leading terms on the 
right-hand side of each of the above equations, since $\sgn(D)$ is a non-square. Thus, the only way that (\ref{sinfor1})-(\ref{sinfor4}) can hold is if $\alpha,\beta,\gamma,\delta \in \finfldq{q}$.  Recall that $\sgn(A) = \sgn(A') = 1$, $\sgn(D) = h$ and thus $\sgn(C) = \sgn(C') = -h/4$.
If we compare the coefficients of $t^{\deg(A)}$ of both sides
of Equations (\ref{A1form})--(\ref{C1form}), we obtain

\begin{eqnarray}
1 & = & \alpha^2 + \left(\frac{-h}{4}\right)\gamma^2,  \label{sgnform1}\\
0 & = &  \alpha \beta + \left(\frac{-h}{4}\right)\gamma \delta , \label{sgnform2}\\
1 & = & \left(\frac{-4}{h} \right) \beta^2 + \delta^2. \label{sgnform3} 
\end{eqnarray}

If $\gamma = 0$, then we deduce as before that $M = \pm I$ or $M = \pm J$, where $J=\left(  \begin{smallmatrix}  1 & 0 \\ 0 & -1\end{smallmatrix}\right)$.  These two cases can only occur when $B = 0$. 

If $\gamma \neq 0$, then from equation (\ref{sgnform2}), we obtain 
\[\delta = \frac{4}{h}\frac{\alpha \beta}{\gamma}.\]
\noindent From equations (\ref{sgnform3}) and (\ref{sgnform1}), it follows that
\[
1 = \left(\frac{-4}{h}\right) \beta^2 + \left(\frac{16}{h^2}\right) \frac{\alpha^2 \beta^2}{\gamma^2} =  \frac{16}{h^2} \frac{\beta^2}{\gamma^2}.\]
\noindent Hence, $\gamma^2 = (16/h^2)\beta^2$ and so $\gamma = \pm (4/h) \beta$.  Write $\gamma = e (-4/h)\beta$ with $e = \pm 1$.  
Then \[\delta = \frac{4}{h} \frac{\alpha \beta}{\gamma} = -e \alpha.\]

\noindent Then by (\ref{sgnform1}),
\[u = \alpha \delta - \beta \gamma = -e \, \alpha^2 - e\, \frac{-h}{4}\gamma^2 = -e \, \left( \alpha^2 + \frac{-h}{4} \gamma^2\right) = -e.\]

\noindent Hence
\[M = \left( \begin{array}{cc} \alpha & \beta \\ (4 u  \beta)/h & u \alpha \end{array}   \right)  \]
with $\alpha,\beta \in \finfldq{q}$, $\alpha^2 - (4/h) \beta^2 =1$ and $u = \det(M) = \pm1$ as desired.  Any change in $u$ between 1 and $-1$ clearly changes the sign of $B'$ if $B' \neq 0$, so $u$ is determined by $\sgn(B')$ in this case.
\end{proof}

We note that with the same notion of reducedness, Theorem \ref{quadeqthm} is not true for real binary quadratic forms, as one cannot deduce that $\alpha,\beta,\gamma,\delta \in \finfldq{q}$ in the same way as in the proof of Theorem \ref{quadeqthm}. 

Part (3) of Theorem \ref{quadeqthm} now yields the following:

\begin{corol}
Any unusual partially reduced binary quadratic form $f=(A,B,C)$ satisfying $|A|=|C|$ is equivalent
to exactly $q+1$ distinct partially reduced quadratic form with the same discriminant. 
\label{qplus1reps}
\end{corol}

\begin{proof}
Consider the equation $\alpha^2 - (4/h)\beta^2 = 1$ in part (3) of Theorem \ref{quadeqthm}.  This equation always has a solution, namely $(\alpha,\beta) = (1,0)$.  
Since $4/h$ is a non-square in $\finfldq{q}^*$, this equation is a non-degenerate conic in $\alpha$ and $\beta$.  Since this conic has at least one solution, it has $q+1$ distinct solutions (see Casse \cite{Casse1}, page 140 or Hirschfeld \cite{Hirsch1}, page 141).  
Each of these $q+1$ solutions yields a binary quadratic form $f'$ equivalent to $f$, as given in Theorem \ref{quadeqthm}.
\end{proof}

\begin{definition}
Let $f=(A,B,C)$ be an imaginary or unusual partially reduced binary quadratic form.
\begin{enumerate}
\item If $|A| <  |C|$, then $f$ is called \emph{reduced}. \label{cond1}
\item If $|A| = |C|$, then $f$ is called \emph{reduced} if it is lexicographically smallest amongst all the partially reduced forms in its equivalence class.
\end{enumerate}
\label{imagRedNew}
\end{definition}

\begin{theorem}
\begin{enumerate}
\item (Existence and Uniqueness) Every imaginary or unusual binary quadratic form over \polyringq{q} is
equivalent to a unique reduced binary quadratic form with the same discriminant.
\item (Finiteness) There are only finitely many reduced imaginary or unusual binary quadratic forms
  with fixed discriminant $D$.
\end{enumerate} 
\label{UnusSummary}
\end{theorem}

\begin{proof}
Part (1) is obvious. To obtain part (2), note that if $f = (A,B,C)$ is reduced, then $|B| < |D| = |AC|$. It is clear that only a finite number of triples $(A,B,C)$ in $\polyringq{q}^3$ can satisfy these conditions. 
\end{proof}

Following Buell \cite{Buell1}, any matrix $M$ such that $f=f\circ M$ is called an \emph{automorphism} of $f$.  Automorphisms of binary quadratic forms need to be considered in the development of the reduction theory of binary cubic forms, and can only occur in very specific cases if a binary quadratic form is reduced:

\begin{theorem}
Suppose $f = (A,B,C)$ is a partially reduced imaginary or  unusual binary quadratic form such that $f = f \circ M$ for some $M \in \genlink{\polyringq{q}}$.  If $|A| < |C|$, then $M = \pm I$, or  $M= \pm J$ if $B=0$, where $J = \left( \begin{smallmatrix} 1 & 0 \\ 0 & -1 \end{smallmatrix} \right)$.     If $|A| = |C|$, then $M = \left( \begin{smallmatrix} \alpha & \beta \\ -u \lambda \beta & u \alpha \end{smallmatrix} \right)$ where $\alpha,\beta \in \finfldq{q}$, $\lambda = -4/h$ and $u := \det(M) = \pm1$.  
If $\beta \neq 0$, then $B \neq 0$, and $\alpha = u\beta(C \lambda - A)/B \in \finfldq{q}$.  Hence, such non-trivial transformations exist if and only if $((C \lambda - A)/B)^2 + \lambda$ is a square in $\finfldq{q}^*$, in which case there are two such non-trivial automorphisms.
\label{distinctalphas}
\end{theorem}

\begin{proof}
Let $M = \left( \begin{smallmatrix} \alpha & \beta \\ \gamma & \delta \end{smallmatrix} \right)$.  Since $f = f \circ M$, we obtain from equation (\ref{sinfor1}), with $f'$ replaced with $f$, that
\[4uA^2 = (2A\alpha + \gamma B)^2 - D\gamma^2. \]

\noindent If $|A| < |C|$, then $|A|^2 < D$, so $\gamma = 0$ and so we obtain from equation (\ref{A1form}) that $\alpha \in \finfldq{q}$ with $\alpha^2 = 1$.  Then, as in the proof of Theorem \ref{quadeqthm}, we obtain $M = \pm I$ or $\pm J$, where the latter two cases force $B=0$.

Now suppose that $|A| = |C|$.  Then $\gamma = (4 u \beta)/h$ and $\delta = u \alpha$ by the proof of Theorem \ref{quadeqthm}, $\alpha^2 - (4/h) \beta^2 = 1$ and $u = \pm 1$ is uniquely determined.  
If $\beta \neq 0$, then $B \neq 0$.  To see this, suppose $\beta \neq 0$ and $B=0$.  Then it follows from equation (\ref{A1form}) that 
\[0 = A(\alpha^2 -1) + C(\lambda^2 \beta^2) = A(-\lambda \beta^2) + C \lambda^2 \beta^2 = \lambda \beta^2 (C\lambda - A),\]
therefore $A = \lambda C$.  Since $D \notin \finfldq{q}$ and $f$ is assumed to be primitive by Proposition \ref{FstMajCub}(4), we have the required contradiction.  

Hence $B \neq 0$.  Again, by equation (\ref{A1form}), we obtain
\begin{eqnarray*} 
0 & = & A(\alpha^2 -1) - Bu\lambda\alpha\beta + C \lambda^2 \beta^2, \\
   & = & -A\lambda \beta^2 - Bu\lambda\alpha\beta + C \lambda^2 \beta^2, \\
   & = &  \lambda \beta \, (\beta (C\lambda - A) - u\alpha B).
\end{eqnarray*}
Therefore, $u \alpha B = \beta (C\lambda - A)$ and solving for $\alpha$ yields $\alpha = u \beta (C\lambda - A)/B$, as desired.  To see that $((C\lambda -A)/B)^2 + \lambda$ is a square in $\finfldq{q}^*$, we substitute $\alpha = u \beta (C\lambda - A)/B$ into $\alpha^2 + \lambda \beta^2 = 1$ to obtain 
\[ [ ((C \lambda - A)/B)^2 + \lambda ] \beta^2 = 1.\]
                                                                                
\noindent So such a non-trivial automorphism exists if and only if $((C \lambda - A)/B)^2 + \lambda$  is a square in $\finfldq{q}^*$. In that case, there are two solutions for $\beta$, and for each of these solutions, there is exactly one solution for $\alpha$, since $u$ is determined by $\sgn(B) \in S$.  This completes the proof.
\end{proof}

We note that if $|A|=|C|$, the case $\beta=0$ is still possible and yields the same four matrices as for $|A|<|C|$.

\section{Reduction of Binary Cubic Forms over \polyringq{q}}\label{sec:cubredn}

In this section, we describe the reduction theory for binary cubic forms.  Once again, this theory allows us to single out a unique representative for each equivalence class of binary cubic forms, in a similar way to the reduction theory for binary quadratic forms.   Throughout this section, $-3D$ is an imaginary or unusual discriminant such that $\sgn(-3D) \in \{1,h\}$.  Furthermore, $h$ and $S$ are as in Section \ref{BQFUnus}.  Recall that our irreducibility assumption forces $ad \neq 0$ for any binary cubic form $f = (a,b,c,d)$.

\begin{definition}
Let $f = (a,b,c,d)$ be a binary cubic form with imaginary or unusual Hessian $H_f = (P,Q,R)$.  
\begin{enumerate}
\item If $H_f$ is imaginary, then $f$ is \emph{reduced} if $H_f$ is reduced, $\sgn(a) \in S$,  and if $Q=0$, then $\sgn(d) \in S$.
\item  If $H_f$ is unusual, then $f$ is \emph{reduced} if $H_f$ is reduced, $\sgn(a) \in S$, if $Q=0$ then $\sgn(d) \in S$, and $f$ is the lexicographically smallest among all equivalent binary cubic forms with the same reduced Hessian $H_f$ that satisfy the previous conditions.
\end{enumerate}
\label{cubimagreduced2}
\end{definition}

The normalization on $\sgn(a)$ is needed because $f$ and $-f$ have the same Hessian.  The normalization of $\sgn(d)$ in the case when $Q = 0$ is required because in this case, $H_f$ has automorphisms  $\pm J$ with $J = \left( \begin{smallmatrix} 1 & 0 \\ 0 & -1 \end{smallmatrix}   \right)$, so $f = (a,b,c,d)$ and $f=(a,-b,c,-d)$ have the same reduced Hessian. 

In order to obtain a unique reduced binary cubic form in the case where the unusual Hessian is non-trivially equivalent to itself, we employ the same technique that was used for unusual binary quadratic forms.  That is, we adopt the convention of choosing the binary cubic form that is the smallest in terms of lexicographical order as specified in Section \ref{BQFUnus}.  Note that it is straightforward to detect whether or not a Hessian has non-trivial automorphisms:  by Theorem \ref{distinctalphas}, it simply requires checking whether or not $(4R/h + P)/Q)^2 - 4/h$ is a square in $\finfldq{q}^*$.  The lexicographical minimization condition eliminates ambiguity in case of non-trivial automorphisms of the Hessian. Collectively, the above conditions ensure uniqueness of reduced representatives:

\begin{theorem}
Every binary cubic form $f$ with imaginary or unusual Hessian is equivalent to a unique reduced binary cubic form with imaginary or unusual Hessian of the same discriminant.  
\label{Unusformequivtoredunus}
\end{theorem}

\begin{proof}
By Proposition \ref{Hessprops2} and Theorem \ref{UnusSummary}, $f$ is equivalent to a binary cubic form of the same discriminant with reduced Hessian, so we can assume without loss of generality that $H_f$ is reduced.

If $\sgn(a) \notin S$, then replace $f$ by $-f = f \circ (-I)$.  If $\sgn(a) \in S$, $Q=0$ and $\sgn(d) \notin S$ then replace $f$ by $f \circ J$, where $J =  \left(\begin{smallmatrix} 1 & 0 \\ 0 & -1  \end{smallmatrix} \right)$.  The resulting form, again denoted by $f$,  has reduced Hessian and satisfies the required normalization conditions on $a$ and $d$.

Finally, suppose that there exists a matrix $M \in \genlink{\polyringq{q}}$ with $H_f \circ M = H_f$ as described in Theorem \ref{distinctalphas}, such that $f \circ M$ also satisfies the required normalization conditions.  Then replace $f$ by the lexicographically smallest among $f$ and $f \circ M$ for all permissible choices of $M$.  The resulting binary cubic form is reduced.  Moreover, none of the above transformations change $D(f)$ or $H_f$.

To obtain uniqueness, suppose $f = (a,b,c,d)$ and $f' = (a',b',c',d')$ are reduced and equivalent.  Then there exists a matrix $M \in \genlink{\polyringq{q}}$ with $f' = f \circ M$.  Set $u = \det(M)$.  Then by Proposition \ref{Hessprops2}, $H_{f'} = u^2 H_f \circ M = H_f \circ(uM)$.   Hence $H_{f'}$ and $H_f$ are equivalent, via the transformation matrix $uM$. By assumption, $H_{f'}$ and $H_f$ are reduced and hence equal by Theorem \ref{UnusSummary}.  By Proposition \ref{Hessprops2}, $H_f$ has discriminant $-3D(f)$. 

Write $H_f = (P, Q, R)$.  If $|P|< |R|$, then by Theorem \ref{distinctalphas}, $uM = I$ or $-I$ or $J$ or $-J$,  where the latter two cases force $Q = 0$. Since $\det(uM) = u^2 \det(M) = u^3$, it follows that $u^3 = 1$ or $-1$.  Thus $a' = au^{-3}$ if $uM = I$ or $J$, and $a' = -au^{-3}$ if $uM = -I$ or $-J.$   If $uM = J$, then $u^3 = \det(uM) = \det(J) = -1$, and hence $a' = au^{-3} = -a$, which contradicts $\sgn(a), \sgn(a') \in S$.  If $uM = -I$, then $u^3 = 1$, and hence $a' = -au^{-3} = -a$, again a contradiction.  If $uM = -J$, then $u^3 = \det(-J) = -1$. In this case, $H_{f'} = H_f$ implies $Q = 0$, so $\sgn(d), \sgn(d') \in S$. But then $f' = u^{-3} f \circ (-J)$ implies $d' = du^{-3} = -d$, a contradiction.  So we must have $uM = I$, and hence $u^3 = 1$. It follows that $f' = u^{-3} f \circ (uM) = f \circ I = f$.   This completes the proof for the case $|P|< |R|$.

If $|P|=|R|$, then the lexicographical minimality of $f$ and $f'$ forces $f = f'$. 
\end{proof}

\section{Coefficient Bounds for Reduced Forms}\label{sec:coefbdds}

We now present bounds on the coefficients of a reduced binary cubic form with imaginary or unusual Hessian.  The motivation for these bounds is to establish that the set of reduced binary cubic forms up to any fixed discriminant degree is in fact finite, which yields a result analogous to part (2) of Theorem \ref{UnusSummary} for cubic forms.  This also ensures that the search procedure in Section \ref{sec:algnum} terminates because the search space is finite.  Moreover, we seek optimal bounds on the coefficients of a reduced binary cubic form so that the search procedure in Section \ref{sec:algnum} is as efficient as possible, as smaller upper bounds give rise to shorter loops.  

The following equality appears in Cremona \cite{Crem1}, and is easily verified by straightforward computation.

\begin{lemma}
Let $f = (a,b,c,d)$ be a binary cubic form with coefficients in
\polyringq{q} with imaginary or unusual Hessian $H_f=(P,Q,R)$.  Let $U = 2b^3 + 27a^2d - 9abc$.
Then \[ 4P^3 = U^2 + 27a^2D.\]
\label{syzygy}
\end{lemma}


\begin{corol} \label{UPprop}
With the notation of Lemma \ref{syzygy}, we have $|U|^2 \leq |P|^3$ and $|a^2 D| \leq |P|^3$.
\label{ImagUnusRedONLY}
\end{corol}

\begin{proof}
Since $-3D$ is imaginary or unusual, there can be no cancellation between the two summands on the right-hand side of the identity of Lemma \ref{syzygy}.  Hence, neither term can exceed $|P|^3$ in absolute value.
\end{proof}

\begin{theorem}
Let $f = (a,b,c,d)$ be a reduced binary cubic
form over \polyringq{q} of discriminant $D$ with imaginary or unusual Hessian $(P,Q,R)$. Then $|a| \leq |D|^{1/4}$, $|b| \leq
|D|^{1/4}$, $|bc| \leq |D|^{1/2}$ and $|ad| \leq |D|^{1/2}$.
\label{acoefbound}
\end{theorem}

\begin{proof}
We have $|Q| < |P| \leq |D|^{1/2}$. By Lemma \ref{syzygy}, $|a|^2 \leq |P|^3/|D| \leq |P| \leq |D|^{1/2}$.

Now an easy computation reveals $U = 2bP - 3aQ$. This, together with $|U| \leq |P|^{3/2}$ and $|aQ| < |aP| \leq |P|^{3/2}$, implies $|bP| \leq \max\{|U|, |aQ|\} \leq |P|^{3/2}$; whence

\begin{equation}
|b| \leq |P|^{1/2} \leq |D|^{1/4}.   \label{eqstar}
\end{equation}

For the remainder of the claim, it suffices to prove only one of $|ad| \leq |D|^{1/2}$ and $|bc| \leq |D|^{1/2}$, as any one of these inequalities, together with $Q = bc - 9ad$, implies the other.

Now by (\ref{eqstar}), $|ac| = |b^2 - P| \leq |D|^{1/2}$.  If $|b| \leq |a|$, then $|bc| \leq |ac| \leq |D|^{1/2}$.  If $|d| \leq |c|$, then $|ad| \leq |ac| \leq |D|^{1/2}$.

Finally, suppose that $|a| < |b|$ and $|c| < |d|$. It is easy to verify that $bR + 3dP - cQ = 0$. Thus, $|cQ| < |dP|$ implies $|dP| = |bR|$. It follows from (\ref{eqstar}) that $|bdP| = |b^2 R| \leq |PR|$, so $|bd| \leq |R|$. This in turn implies $|c|^2 = |3bd + R| \leq |R|$, and hence again from (\ref{eqstar}), $|bc|^2 \leq |PR| = |D|$, as claimed. 
\end{proof}

These bounds are indeed sharp.  An example of a reduced binary quadratic form over $\polyringq{5}$ where $|a|=|b|=|c|=|d| = |D|^{1/4}$ is $f=(a,b,c,d) = (2t+4,3t+4,3t+3,3t+1)$, where $D = t^4 + 4t^3 + t^2 + 3$.

We use the bounds of Theorem \ref{acoefbound} for our tabulation algorithm.  Specifically, we loop over all $a, b, c, d$ satisfying these bounds. An upper bound on $|abcd|$ determines how often the inner most loop is entered.  By Theorem \ref{acoefbound}, such a bound is given by $|D|$. Hence, $|D|$ is an upper bound on the number of forms of discriminant $D$ that the algorithm checks for membership in the Davenport-Heilbronn set $\newU$.

\section{The Davenport-Heilbronn Theorem}\label{sec:DH}

We now briefly discuss the Davenport-Heilbronn theorem for function fields.  The original Davenport-Heilbronn theorem \cite{HeDav2} states
that there exists a discriminant-preserving bijection from a certain
set $\newU$ of equivalence classes of integral binary cubic forms
of discriminant $D$ to the set of $\rats$-isomorphism classes of
cubic fields of the same discriminant $D$. Therefore, if one can
compute the unique reduced representative $f$ of any class of
forms in $\newU$ of discriminant $D$ with $|D| \leq X$, then this
leads to a list of minimal polynomials $f(x, 1)$ for all cubic
fields of discriminant $D$ with $|D| \leq X$.

The situation for cubic function fields is completely analogous.
We now state the function field version of the
Davenport-Heilbronn theorem, describe the Davenport-Heilbronn set $\newU$ for function
fields, and provide a fast algorithm for
testing membership in $\newU$ that is in fact more efficient than
its counterpart for integral forms.

For brevity, we let $[f]$ denote the equivalence class of any
primitive binary cubic form $f$ over \polyringq{q}. Fix any
irreducible polynomial $p \in \polyringq{q}$. Analogous to \cite{Bela1,Bela2,Coh2}, we define $\newV_p$
to be the set of all equivalence classes $[f]$ of binary cubic
forms such that $p^2 \nmid D(f)$. In other words, if $D(f)
= i^2 \Delta$ where $\Delta$ is square-free, then $f \in
\newV_p$ if and only if $p \nmid i$. Hence, $f \in \bigcap_p \newV_p$ if
and only if $D(f)$ is square-free.

Now let $\newU_p$ be the set of equivalence classes $[f]$ of
binary cubic forms over \polyringq{q} such that
\begin{itemize}
\item either $[f] \in \newV_p$, or
\item $f(x,y) \equiv \lambda (\delta x - \gamma y)^3  \pmod{p}$
for some $\lambda \in \polyringq{q}/(p)^*$, $\gamma,\delta \in
\polyringq{q}/(p)$, 
and
in addition, $f(\gamma,\delta) \not \equiv 0 \pmod{p^2}$.
\end{itemize}

For brevity, we summarize the condition $f(x,y) \equiv \lambda
(\delta x - \gamma y)^3 \pmod{p(t)}$ for some $\gamma,\delta \in
\polyringq{q}/(p)$ and $\lambda \in \polyringq{q}/(p)^*$ with the
notation $(f,p) = (1^3)$ as was done in \cite{HeDav1,HeDav2}.

Finally, we set 
$\newU = \bigcap_p \newU_p$.  The set \newU\ is the set under
consideration in the Davenport-Heilbronn theorem.
The version given below appears in \cite{Pieter2}.  A more
general version of this theorem for Dedekind domains appears in
Taniguchi \cite{Taniguchi}.

\begin{theorem} \label{DHThm}
Let $q$ be a prime power with $\gcd(q, 6) = 1$. Then there exists
a discriminant-preserving bijection between
$\ratsfuncq{q}$-isomorphism classes of cubic function fields and
classes of binary cubic forms over \polyringq{q}\ belonging to
$\newU$.  This bijection maps a class $[f] \in \newU$ to the triple of $\ratsfuncq{q}$-isomorphic cubic fields that have minimal polynomial $f(x,1)$.
\end{theorem}

In order to convert Theorem \ref{DHThm} into an algorithm, we
require a fast method for testing membership in the set $\newU$.
This is aided by the following efficiently testable conditions:

\begin{propn} \label{Uprops}
Let $f = (a,b,c,d)$ be a binary cubic form over \polyringq{q} with
Hessian $H_f = (P,Q,R)$. Let $p \in \polyringq{q}$ be irreducible.
Then the following hold:
\begin{enumerate}
\item $(f,p) = (1^3)$ if and only if $p \mid \gcd (P,Q,R)$.
\item If $(f,p) = (1^3)$ then $f \in \newU_{p}$ if and only if $p^3 \nmid D(f)$.
\end{enumerate}
\end{propn}

In addition, classes in $\newU$ contain only irreducible forms;
this result can be found for integral cubic forms in Chapter 8 of \cite{Coh2},
and is completely analogous for forms over \polyringq{q}. 

\begin{theorem} \label{imaginuirr}
Any binary cubic form whose equivalence class belongs to $\newU$
is irreducible.
\end{theorem}

By Theorem \ref{DHThm}, if $[f] \in \newU$, then $f(x, 1)$
is the minimal polynomial of a cubic function field over
$\ratsfuncq{q}$. This useful fact eliminates the necessity for a
potentially costly irreducibility test when testing membership
in~$\newU$.

Using Proposition \ref{Uprops}, we can now formulate an algorithm
for testing membership in $\newU$. This algorithm will be used in our tabulation routines for cubic function fields.  The algorithm is slightly different from the one in \cite{ANTSPie}; Hessian and discriminant values are passed into the routine, rather than computed inside Algorithm \ref{cubformtestversion2}.

\begin{algorithm}
\caption{Test for membership in $\newU$}
\label{cubformtestversion2}
\begin{algorithmic}[1]
\REQUIRE A binary cubic form $f=(a,b,c,d)$ with coefficients in \polyringq{q}, its Hessian $(P,Q,R)$ and its discriminant $D$.
\ENSURE \texttt{true} or \texttt{false} according to
whether or not $[f]$ lies in the set \newU.\\
\STATE Compute $\ell_H := \gcd(P,Q,R)$.
\IF{$\ell_H$ is not square-free}  
\RETURN \texttt{false};
\ENDIF
\STATE Compute $s:=-3D/(\ell_H)^2$;
\IF{$\gcd(s,\ell_H) \neq 1$}  
\RETURN \texttt{false};
\ENDIF
\IF{$s$ is square-free}
\RETURN  \texttt{true};
\ELSE
\RETURN \texttt{false};
\ENDIF
\end{algorithmic}
\end{algorithm}
We note here that step 2 of Algorithm \ref{cubformtestversion2} uses part (2) of Proposition \ref{Uprops}.  In step 5, note that if $p \mid \ell_H$ and $p \mid s$, then $p^3 \mid D$ and using part (1) of Proposition \ref{Uprops} yields $[f] \notin \newU_p$.  If $f$ passes steps 1-6, then $s$ is not square-free if and only if there exists an irreducible polynomial $p \in \polyringq{q}$ with $p^2 \mid s$ and hence $p \nmid \ell_H$.  Using part (2) of Proposition \ref{Uprops} again, this rules out $(f,p)=(1^3)$. On the other hand, we
also have $p^2 \mid D(f)$, so $f \notin \newV_p$, and hence $f \notin \newU_p$.  Thus, if $s$ is not  square-free then $[f] \not\in \newU_p$ for some $p$, or equivalently, $[f] \not\in \newU$.  Conversely, if $s$ is square-free, then the primes $p$ dividing $s$, and hence dividing $D$, occur in $D$ to the first power.  Thus $f \in \newV_p$ for all such $p$, proving the validity of step 7.

Note that steps 2 and 7 of Algorithm \ref{cubformtestversion2} require
tests for whether a polynomial $F \in \polyringq{q}$ is
square-free. This can be accomplished very efficiently with a
simple gcd computation, namely by checking whether $\gcd(F, F') =
1$, where $F'$ denotes the formal derivative of $F$ with respect
to $t$. This is in contrast to the integral case, where square-free
testing of integers is generally difficult; in fact, square-free
factorization of integers is just as difficult as complete
factorization. Hence, the membership test for $\newU$ is more
efficient than its counterpart for integral forms.  One may be tempted to try a more na\"{i}ve approach, namely factoring the square part $i$ out of the discriminant $D$ and then testing only the resulting $p$.  Even though factorization of polynomials over finite fields is much easier than factoring integers, this would add an unnecessary $\log$ factor to our theoretical run times, so we did not attempt this approach.

\section{Algorithms and Numerical Results}\label{sec:algnum}

We now describe the tabulation algorithm for cubic function
fields corresponding to reduced binary cubic
forms over \polyringq{q} with imaginary or unusual Hessian; that is, cubic extensions of
\ratsfuncq{q} of discriminant $D$ where 
$\deg(D)$ is odd, or $\deg(D)$ is even and $\sgn(-3D)$ is a non-square in
$\finfldq{q}^*$.

The idea of the algorithm is as follows. Input a prime power $q$
coprime to 6, a degree bound $B \in \mathbb{N}$, a primitive root $h$ of $\finfldq{q}^*$ and the set $S$. The algorithm outputs minimal polynomials for all \ratsfuncq{q}-isomorphism
classes of cubic extensions of \ratsfuncq{q} of discriminant $D$
such that $-3D$ is imaginary or unusual and $|D| \leq q^B$.  The
algorithm searches through all coefficient 4-tuples $(a,b,c,d)$
that satisfy the degree bounds of Theorem \ref{acoefbound} with
$|D|$ replaced by $q^B$ such that the form $f = (a,b,c,d)$ satisfies
the following conditions:

\begin{enumerate}
\item $f$ is reduced;
\item $f$ has imaginary (resp.\ unusual) Hessian;
\item $f$ belongs to an equivalence class in $\newU$;
\item $f$ has a discriminant $D$, where $\deg(D) \leq B$.
\end{enumerate}
If $f$ passes all these tests, the algorithms outputs $f(x,1)$
which by Theorem \ref{DHThm} is the minimal polynomial of a triple
of \ratsfuncq{q}-isomorphic cubic function fields of discriminant
$D$.  

The test of whether or not $f$ is reduced in the case where $H_f$ is unusual is more
involved than in the imaginary case.  Recall from Theorem \ref{quadeqthm} and Corollary \ref{qplus1reps} that if $H_f = (P,Q,R)$ is the Hessian of $f$ and both $|P|= |R|=\sqrt{|D|}$ and $\sgn(P)=1$ are satisfied, then this test requires the computation of $q+1$ partially reduced binary quadratic forms equivalent to $H_f$, along with the determination of whether $H_f$ is the smallest lexicographically amongst the $q+1$ quadratic forms.  Furthermore, by Theorem \ref{distinctalphas}, $H_f$ may have non-trivial automorphisms, potentially resulting in more than one equivalent binary cubic form with the same Hessian.  Fortunately, the proportion of forms for which these extra checks need to be performed is small by Lemma \ref{l:forms2} below, so this does not significantly affect the run time of the tabulation algorithm. However, the fact that the bounds in Theorem \ref{acoefbound} can be attained when $|P| = |R|$, whereas they can not be exactly met when $|P| < |R|$, make the algorithm for unusual Hessians slower by a factor of $q$, as seen in Corollary \ref{algislinanalysisunus}.

Checking whether a binary cubic form has a partially reduced Hessian amounts to checking the conditions specified in Definition \ref{cubimagreduced2}.  Checking whether a given cubic form lies in $\newU$ involves running Algorithm \ref{cubformtestversion2}.   A basic version of Algorithm \ref{tabimagbetter} loops over each of the coefficients $a,b,c,d$ up to the bounds given in Theorem \ref{acoefbound}.  We omit the description here, instead giving an improved version in Algorithm \ref{tabimagbetter}.  
First, we modify the \emph{for} loop on $a$ so that the condition $\sgn(a) \in S$ is checked first.  
The other improvement involves the \emph{for} loop on $d$.  Instead of using the \emph{for} loop on $d$ as one might do in a basic implementation, we determine which degree values of $d$ lead to an odd (resp.\ even) degree discriminant.  This entails computing the quantities $m_1$ through $m_5$ and the maximum of these values $m$.  The values of $m_1$ through $m_5$ are simply the degree values of each term in the formula for the discriminant of a binary cubic form.  If the maximum value $m$ of these terms is taken on by a unique term amongst the $m_i$ and $m$ is not of the appropriate parity, the next degree value for $d$ is considered instead of proceeding further with computing the Hessian and discriminant.  This allows us to avoid extra computations for discriminants that do not have odd (resp.\ even) degree altogether.

For admissible values, we compute the Hessians and the discriminant as before, but we can compute the quantities $P$ and $t_1:= bc ,t_2:= c^2$ before any information about $d$ is known.   We also note that the seemingly redundant check of the conditions $\deg(D) \leq B$ and $\deg(D)$ odd (resp.\ even) near the end of the algorithm may be necessary in the event of cancellation of terms when $m$ is not taken on by a unique term among the $m_i$ (since it may be possible that the maximum $m$ of the degrees on the terms of the cubic discriminant satisfies $m > B$ but $\deg(D) \leq B$).

\begin{algorithm}
\caption{Tabulation of Imaginary (resp.\ Unusual) Cubic Function Fields (modified version)}
\label{tabimagbetter}
\begin{algorithmic}[1]
\REQUIRE A prime power $q$ not divisible by $2$ or $3$, a primitive root $h$ of $\finfldq{q}$, the set $S=\{1,h,h^2,\ldots h^{(q-3)/2} \}$, and a positive integer $B$.
\ENSURE Minimal polynomials for all \ratsfuncq{q}-isomorphism classes of cubic function fields of discriminant $D$ with $\deg(D)$ odd, $\sgn(-3D) \in \{1,h\}$ (resp.\ $\deg(D)$ even and $\sgn(-3D)=h$), and $\deg(D) \leq B$.\\
 
\FOR{$\deg(a) \leq B/4$ \textbf{AND} $\sgn(a) \in S$}
\FOR{$\deg(b) \leq B/4$}
\FOR{$\deg(c) \leq B/(2\deg(b))$}
\STATE $m_1: = 2 (\deg(b) + \deg(c))$;
\STATE $m_2:= \deg(a) + 3 \deg(b)$
\FOR{$i = 0$ to $B/2 - \deg(a)$}
\STATE $m_3: = \deg(a) + \deg(b) + \deg(c) + i$;
\STATE $m_4:= 3 \deg(b) + i$
\STATE $m_5:= 2(\deg(a) + i)$
\STATE $m:= \max \{m_1,m_2,m_3,m_4,m_5\}$
\IF{($m$ is not taken on by a unique term among the $m_i$) \textbf{OR} ($m$ is taken on by a unique term \textbf{AND} $m$ is odd (resp.\ even) \textbf{AND} $q^m \leq q^B$)}
\STATE Compute $P:= b^2 - 3ac$;
\STATE Compute $t_1:= bc$;
\STATE Compute $t_2:= c^2$
\FOR{$\deg(d) = i$}
\STATE Set $f:= (a,b,c,d)$;
\STATE Compute $Q := t_1- 9ad$;
\STATE Compute $R := t_2 - 3bd$;
\STATE Compute $-3D = -3D(f) = Q^2 - 4PR$;
\IF{$\deg(D)$ is odd and  $\sgn(-3D) \in \{1,h\}$ (resp.\ $\deg(D)$ even and $\sgn(-3D)=h$)  \textbf{AND}
$\deg(D) \leq B$
\textbf{AND}  $f$ is reduced  \textbf{AND} $[f] \in \newU$}  
\STATE  Output $f(x,1)$;
\ENDIF 
\ENDFOR
\ENDIF  
\ENDFOR
\ENDFOR
\ENDFOR
\ENDFOR
\end{algorithmic}
\end{algorithm}

Some extra routines are needed for the unusual Hessian case and are described here.  First, a routine called \emph{ConicSolver} is used to solve the equation $\alpha^2 - (4/h) \beta^2 =1$ for $\alpha,\beta \in \finfldq{q}$ via brute force.  Each solution pair is stored in an array.  If the Hessian $H_f = (P,Q,R)$ of a binary cubic form $f$ satisfies $|P|=|R|$ and $\sgn(P)=1$, then the solutions $(\alpha,\beta)$ are used as in Corollary \ref{qplus1reps}.  Each pair $(\alpha,\beta)$ determines a matrix $M_{\alpha,\beta} = \left( \begin{smallmatrix} \alpha & \beta \\ 4u \beta/h & u \alpha \end{smallmatrix} \right)$, where $u=\pm1$ is chosen so that either $\sgn(Q_i) \in S$ when $Q_i \neq 0$ for any of the $q+1$ Hessians $(P_i,Q_i,R_i)$ of the corresponding partially reduced $q+1$ cubic forms, or $\sgn(d_i) \in S$ when $Q_i=0$; here, $f_i = (a_i,b_i,c_i,d_i) = f \circ M$ where $(P_i, Q_i, R_i) = H_{f_i} = H_f \circ M$, and $M$ is the appropriate matrix $M_{\alpha, \beta}$.  Each of these matrices is applied to each Hessian $H_f$ under consideration, with the value $u$ determined as described above. 

If the Hessian $H_f$ is the smallest in terms of lexicographical order amongst the $q+1$ Hessians computed, then the corresponding binary cubic form $f$ is output if it lies in $\newU$.  This task is accomplished via a routine called \emph{IsSmallestQuad}, which returns \texttt{true} if $H_f$ is the smallest lexicographically amongst the $q+1$ forms equivalent to $H_f$, and returns \texttt{false} otherwise.  In the event that $H_f$ is the smallest and $H_f$ has non-trivial automorphisms as given in Theorem \ref{distinctalphas}, then $f$ is tested to see if it is the smallest in terms of lexicographical ordering amongst itself and the two binary cubic forms equivalent to $f$ via a non-trivial automorphism.  This task is accomplished via a routine called \emph{IsDistCub}, which returns \texttt{true} if $f$ is the smallest lexicographically amongst these binary cubic forms, and returns \texttt{false} otherwise.  If the routine \emph{IsDistCub} returns \texttt{true} and the binary cubic form $f$ lies in $\newU$, the minimal polynomial $f(x,1)$ is output.

If the (unusual) Hessian $H_f$ of a binary cubic form $f$ satisfies $|P| < |R|$, then the algorithm is much simpler, since the partially reduced binary cubic form is in fact reduced.  That is, we simply test the binary cubic form $f$ to see if it lies in \newU, just like in the case of imaginary Hessians.  If it does, the minimal polynomial $f(x,1)$ is output.

Tables \ref{tabtableOdd} and \ref{tabtableEven} present the results of our computations using Algorithm \ref{tabimagbetter} for cubic function fields with imaginary and unusual Hessian, respectively, for
various values of $B$, 
given in column 2 of each table.  The third column gives the number of fields with discriminant $D$ satisfying $\deg(D) \leq B$ and the fourth column denotes the number of reduced cubic forms whose Hessians have non-trivial automorphisms.  The last column gives the various timings for each degree bound.  The results in these tables extend and correct those in \cite{ANTSPie}.    The lists of cubic function fields were computed on a multi-processor machine with four 2.8 GHz Pentium 4 processors running Linux with 4 GB of RAM, with a 180 MB cache using the C++ programming language coupled
with the number theory library NTL \cite{Shoup1} to implement our algorithm.  The results for unusual discriminants in this table are new and did not appear in \cite{ANTSPie}. 

\begin{center}
\begin{longtable}{|c|c|c|c|c|} 
\caption{Number of cubic function fields over $\ratsfuncq{q}$ of discriminant $D$ with $\deg(D) \leq B$, $\deg(D)$ odd}\label{tabtableOdd} \\

\hline \multicolumn{1}{|c|}{\,\,$q$\,\,} & \multicolumn{1}{|c|}{Degree bd. $B$} & \multicolumn{1}{|c|}{\# of fields} & \multicolumn{1}{c|}{\# non-triv. auto.} & 
\multicolumn{1}{c|}{Total elapsed time} \\ \hline \endfirsthead

\multicolumn{5}{c}%
{{\bfseries \tablename\ \thetable{} -- continued from previous page}} \\
\hline \multicolumn{1}{|c|}{\,\,$q$\,\,} & \multicolumn{1}{|c|}{Degree bd. $B$} & \multicolumn{1}{|c|}{\# of fields} & \multicolumn{1}{c|}{\# non-triv. auto.} & 
\multicolumn{1}{c|}{Total elapsed time} \\ \hline
\endhead

\hline \multicolumn{5}{|c|}{{Continued on next page}} \\ \hline
\endfoot

\hline 
\endlastfoot
\,\,$5$\,\, & $3$ & $100$ & --- & 0.02 seconds \\ \cline{2-5}
& $5$ & 2100 & --- & 1.21 seconds \\ \cline{2-5}
& $7$ & 64580 & --- & 31.66 seconds \\ \cline{2-5}
& $9$ & 1877260  & --- & 26 minutes, 4 sec \\ \cline{2-5}
& $11$ & 45627300 & --- & 9 hours, 31 min, 45 sec \\ \cline{2-5}
\hline
\,\,$7$\,\, & $3$ & 294 & --- & 0.17 seconds\\ \cline{2-5}
& $5$  & 12642 & ---  & 17.48 seconds  \\ \cline{2-5}
& $7$  & 718494 & --- & 25 minutes, 4 sec \\ \cline{2-5}  
& $9$  & 39543210 & --- & 21 hours, 56 min, 45 sec \\ \cline{2-5}
\hline
\,\,$11$\,\, & $3$  & 1210 & --- & 2.61 seconds \\ \cline{2-5}
& $5$  & 134310 & --- & 18 minutes, 45 sec  \\ \cline{2-5}
&$7$  & 17849810 & ---& 1 day, 2 hours, \\ 
&&&&9 min, 28 sec  \\ \hline
\,\,$13$\,\, & $3$  & 2028 & --- & 7.10 sec \\ \cline{2-5}
& $5$  & 318396 & --- & 55 minutes, 54 sec \\  \cline{2-5}
& $7$  & 58239948  & --- & 6 days, 31 min, 56 sec   \\ \hline
\end{longtable}
\end{center}
\begin{center}
\begin{longtable}{|c|c|c|c|c|} 
\caption{Number of cubic function fields over $\ratsfuncq{q}$ of discriminant $D$ with $\deg(D) \leq B$, $\deg(D)$ even}\label{tabtableEven} \\

\hline \multicolumn{1}{|c|}{\,\,$q$\,\,} & \multicolumn{1}{|c|}{Degree bd. $B$} & \multicolumn{1}{|c|}{\# of fields} & \multicolumn{1}{c|}{\# non-triv. auto.} & 
\multicolumn{1}{c|}{Total elapsed time} \\ \hline \endfirsthead

\multicolumn{5}{c}%
{{\bfseries \tablename\ \thetable{} -- continued from previous page}} \\
\hline \multicolumn{1}{|c|}{\,\,$q$\,\,} & \multicolumn{1}{|c|}{Degree bd. $B$} & \multicolumn{1}{|c|}{\# of fields} & \multicolumn{1}{c|}{\# non-triv. auto.} & 
\multicolumn{1}{c|}{Total elapsed time} \\ \hline
\endhead

\hline \multicolumn{5}{|c|}{{Continued on next page}} \\ \hline
\endfoot

\hline 
\endlastfoot
\,\,$5$\,\, 
& $4$ & 280 & 10 & 0.88 seconds\\ \cline{2-5}
& $6$ & 6480 & 10 & 19.06 seconds \\ \cline{2-5}
& $8$ &156920 & 320 & 12 minutes, 3 sec \\ \cline{2-5}
& $10$ & 4688440 & 320 & 7 hours, 12 min, 32 sec \\ \cline{2-5}
& $12$ & 117981240 & 11385 & 10 days, 19 hours, \\ 
& & & & 3 min, 8 sec \\ \hline
\,\,$7$\,\, 
&$4$  & 1077 & 42 & 12.18 seconds\\ \cline{2-5}
&$6$  & 51645  & 42 & 17 minutes, 28 sec  \\ \cline{2-5}
&$8$  & 2475271 & 2436 & 10 hours, \\ 
&&&&38 minutes, 33 sec \\ \cline{2-5}
&$10$ &   138360895 & 2436  &  23 days, 4 hours,  \\ 
&&&& 16 min, 51 sec \\ \hline
\,\,$11$\,\, & $4$  & 5722 & 54 & 13 minutes, 5 sec \\ \cline{2-5}
& $6$  & 810372 & 54 & 21 hours,  \\ 
& & & & 54 min, 59 sec \\ \hline
\,\,$13$\,\, & $4$  & 11334 & 304 & 40 minutes, 42 sec \\ \cline{2-5}
& $6$  & 2240106 & 304 & 3 days, 14 hours  \\
\end{longtable}
\end{center}

Timings for Algorithm \ref{tabimagbetter} are compared to those of a basic version of Algorithm \ref{tabimagbetter} in Table \ref{ModvsBasicimag} for $q=5,7$.  Recall that this basic version simply loops over all all $a,b,c,d$ satisfying the bounds of Theorem \ref{acoefbound}, without a priori eliminating potential unsuitable values of $d$ as was done in Algorithm \ref{tabimagbetter}.  The cases $q=11,13$ are omitted for brevity.  As seen in the third column of these tables, the modified algorithm is a significant improvement over the basic algorithm.  These improvements appear to get better as the degree increases, likely because higher degrees give fewer chances of ``bad" leading term cancellations in the discriminant, thereby a priori eliminating more unsuitable values of $d$.  Also, interestingly, the improvement seems more pronounced for even degrees. Again, this is likely due to fewer ``bad" cancellations of leading terms of the discriminant in this case. For example, the terms $(bc)^2$ or $27(ad)^2$ in $D$ can never individually dominate if $\deg(D)$ is odd.

\begin{center}
\begin{longtable}{|c|c|c|c|c|} 
\caption{Basic vs. modified algorithm timings}\label{ModvsBasicimag} \\

\hline \multicolumn{1}{|c|}{\,\,$q$\,\,} & \multicolumn{1}{|c|}{Degree bd.} & \multicolumn{1}{|c|}{Basic times} & \multicolumn{1}{c|}{Modified times} & 
\multicolumn{1}{c|}{Basic/Mod} \\ \hline \endfirsthead

\multicolumn{5}{c}%
{{\bfseries \tablename\ \thetable{} -- continued from previous page}} \\
\hline \multicolumn{1}{|c|}{\,\,$q$\,\,} & \multicolumn{1}{|c|}{Degree bd.} & \multicolumn{1}{|c|}{Basic times} & \multicolumn{1}{c|}{Modified times} & 
\multicolumn{1}{c|}{Basic/Mod} \\ \hline
\endhead

\hline \multicolumn{5}{|c|}{{Continued on next page}} \\ \hline
\endfoot

\hline 
\endlastfoot
\,\,$5$\,\, & $3$  & 0.09 seconds & 0.02 seconds & 4.5 \\ \cline{2-5}
& $4$  & 7.67 seconds  & 0.88 seconds  & 8.72 \\  \cline{2-5}
&$5$  & 7.19 seconds & 1.21 seconds & 5.94 \\ \cline{2-5}
&$6$  & 3 minutes, 55 sec  & 19.06 seconds & 12.32 \\\cline{2-5}
&$7$  & 3 minutes,   20 sec & 31.66 seconds & 6.33 \\ \cline{2-5}
&$8$  & 5 hours, 52 min, 38 sec  & 12 minutes, 3 sec & 29.28   \\ \cline{2-5}
&$9$  & 4 hours,   & 26 minutes, 4 sec & 9.51\\ 
&          &    7 min, 57 sec     &              &    \\ \cline{2-5}
&$10$ & 5 days, 1 hour, & 7 hours,  &  16.87 \\ 
 &          &  38 min, 6 sec  &   12 min, 32 sec&  \\ \cline{2-5}
&$11$  & 6 days, 8 hours,   & 9 hours, & 15.99\\ 
&            &  22 min, 7 sec     &  31 min, 45 sec   &  \\ \hline
\,\,$7$\,\, & $3$  & 0.73 seconds  & 0.17 seconds & 4.29 \\ \cline{2-5}
& $4$  & 2 minutes,  1 sec &  12.18 seconds & 9.92 \\   \cline{2-5}
& $5$  & 3 minutes, 56 sec & 17.48 seconds & 13.49 \\\cline{2-5}
&$6$  & 3 hours,  & 17 minutes, 28 sec  & 11.26 \\   
&         &     16 min, 40 sec      &       &    \\ \cline{2-5}
& $7$  & 3 hours, 12 min, 34 sec & 25 minutes, 4 sec & 7.68 \\ \cline{2-5}  
& $8$  & 10 days,  & 10 hours,   & 22.61 \\ 
&          &  38 min, 17 sec   &  38 minutes, 33 sec &   \\ \cline{2-5}
& $9$  & 12 days, 18 hours,  & 21 hours,  & 13.97 \\  
&         &  35 min, 3 sec  & 56 min, 45 sec  &  \\  \cline{2-5}
\end{longtable}
\end{center}

The value $h=2$ was chosen as a primitive root for $\finfldq{5}$, $\finfldq{11}$ and $\finfldq{13}$.  For $\finfldq{7}$, $h=3$ was chosen.  This completely determines the set $S$ specified in Section \ref{BQFUnus}.  These sets were $\{1,2\}$, $\{1,2,3\}$, $\{1,2,4,5,8\}$ and $\{1,2,3,4,6,8\}$ for $\finfldq{5}$, $\finfldq{7}$, $\finfldq{11}$ and $\finfldq{13}$, respectively.

The worst-case complexity of the algorithm, expressed in terms of the number of field operations required as a function of the degree bound $B$, largely depends on the size of each of the coefficients $a,b,c,d$ of a binary cubic form that the algorithm loops over.  It follows from Theorem \ref{acoefbound} that $|abcd| \leq |D|$, so the number of forms that need to be checked up to an upper bound $B$ on $\deg(D)$ is of order $B$.  This idea is fully explained in the following lemmas.

\begin{lemma} \label{l:forms1}
For $s \in \mathbb{N}$, denote by $\mathcal{F}_s$ the set of binary cubic forms $f = (a, b, c, d)$ over $\Fq[t]$ such that $\deg(D(f)) = s$, $\deg(a) \leq s/4$, $\deg(b) \leq s/4$, $\deg(ad) \leq s/2$, $\deg(bc) \leq s/2$, and $\sgn(a) \in S$. Then

\[ \#\mathcal{F}_s \leq \begin{cases} \ \displaystyle \frac{q^3}{32} \, s^2 q^s + O(s q^s) & \mbox{if $s$ is odd} \ , \\[10pt] \ \displaystyle \frac{q^4}{32} \, s^2 q^s + O(s q^s) & \mbox{if $s$ is even} \ ,  \end{cases} \]
as $s \rightarrow \infty$.
\end{lemma}
\begin{proof}
The number of monic polynomials in $\Fq[t]$ of degree $m$ is $q^m$. Hence, the number of pairs of monic polynomials $(G, H)$ with $\deg(G) \leq s/4$ and $\deg(GH) \leq s/2$ is
\begin{eqnarray*}
N_s & = & \sum_{m \leq s/4} \ \sum_{m+n \leq s/2} q^{m+n} = \sum_{m=0}^{s/4} q^m \, \sum_{n = 0}^{\lfloor s/2 \rfloor -m} q^n \\
& = & \frac{1}{q-1} \, \sum_{m=0}^{s/4} (q^{\lfloor s/2 \rfloor +1} - q^m) = \frac{q}{4(q-1)} s q^{\lfloor s/2 \rfloor } + O(q^{s/2})
\end{eqnarray*}
as $s \rightarrow \infty$.

Let $f = (a,b,c,d) \in \mathcal{F}_s$. Note that $ad \neq 0$. Thus there are $(q-1)/2$ choices for $\sgn(a)$ and $q-1$ choices for $\sgn(d)$, for a total of $N_s(q-1)^2/2$ pairs $(a, d)$. Similarly, the possible number of $(b,c)$ pairs is $N_s q^2$. Hence the total number of forms in $\mathcal{F}_s$ is
\[ N_s^2 \, \frac{q^2(q-1)^2}{2} = \frac{q^4}{32} \, s^2 q^{2 \lfloor s/2 \rfloor} \ . \]
The result now simply follows from the fact that $2 \lfloor s/2 \rfloor = s$ if $s$ is even and $s-1$ if $s$ is odd.
\end{proof}

\begin{lemma} \label{l:forms2}
Let $\mathcal{F}_s$ be as defined in Lemma \ref{l:forms1} and $E_s$ denote the number of forms in $\mathcal{F}_s$ such that $H_f = (P, Q, R)$ is partially reduced and $\deg(P) = \deg(R)$. Then
\[ E_s \leq \frac{q+1}{2} \, q^s \]
if $s \equiv 0 \pmod{4}$, and $E_s = 0$ if $s \not \equiv 0 \pmod{4}$.
\end{lemma}

\begin{proof}
Let $f = (a,b,c,d) \in \mathcal{F}_s$ so that $H_f$ is partially reduced with $\deg(P) = \deg(R)$. Then $\deg(Q) < \deg(P) = \deg(R) = s/2$, so $s$ must be even. A straightforward calculation yields
\begin{equation} \label{eq:lin-comb}
bQ = cP + 3aR \ , \qquad 3cQ = 3dP + bR \ ,
\end{equation}
whence $\deg(c) = \deg(bQ-3aR) - \deg(P) \leq s/4$ and $\deg(d) = \deg(3cQ-bR) - \deg(P) \leq s/4$. Hence all of $a, b, c, d$ have degree no more than $s/4$. In fact, the formulas for $Q, P, R$ can be seen to force $\deg(a) = \deg(c) = s/4$ or $\deg(b) = \deg(d) = s/4$. It follows that $E_s = 0$ unless $s$ is a multiple of 4, which we assume for the remainder of the proof.

If $\deg(b), \deg(d) \leq s/4 - 1$, then $\sgn(c)^2 = \sgn(R) = -h/4$ forces $-1$ to be a non-square in $\Fq$, or equivalently, $q \equiv -1 \pmod{4}$. Together with $-3\, \sgn(ac) = \sgn(P) = 1$ and $\sgn(a) \in S$, this determines $\sgn(a)$ and $\sgn(c)$ uniquely. Thus, this accounts for at most $q^s$ forms if $q \equiv -1 \pmod{4}$ and no forms if $q \equiv 1 \pmod{4}$.

If $\deg(a), \deg(c) \leq s/4-1$, then $\deg(b) = \deg(d) = s/4$. In this case, $P = b^2 - 3ac$ forces $\sgn(b) = \pm 1$, so there are $2q^{s/4}$ possibilities for $b$. Then $\sgn(R) = -3\, \sgn(bd)$ determines $\sgn(d)$, so there are $q^{s/4}$ choices for $c$ and $d$ each. The number of permissable $a$ is at most
\[ \frac{q-1}{2} \, \sum_{m=0}^{s/4-1} q^m < \frac{q^{s/4}}{2} \ . \]
So this case produces no more that $q^s$ forms.

Finally, suppose that $a, b, c, d$ all have degree $s/4$. By (\ref{eq:lin-comb}), the leading coefficients of $a$ and $b$ determine those of $c$ and $d$ uniquely. Specifically,  $\sgn(c) = 3h\, \sgn(a)/4$, and substituting this into $P - b^2 - 3ac$ yields $\sgn(b)^2 - h (3 \, \sgn(a)/2)^2 = 1$. The equation $v^2 - hu^2 = 1$ represents a non-degenerate conic which has $q+1$ solutions $(u,v)$ over $\Fq$. However, since $a \neq 0$, the two solutions $(u, v) = (0, \pm 1)$ are invalid. In addition, if $q \equiv -1 \pmod{4}$, then $v = 0$ produces two invalid solutions. In fact, $\sgn(a) \in S$ eliminates half of the remaining solutions, allowing $(q-3)/2$ choices for $(\sgn(a), \sgn(b))$ when $q \equiv -1 \pmod{4}$, and $(q-1)/2$ choices when $q \equiv 1 \pmod{4}$. So the number of possibilities for $f$ is $(q-3)q^s/2$  if $q \equiv -1 \pmod{4}$ and $(q-1)q^s/2$  if $q \equiv 1 \pmod{4}$.

In all scenarios, the number of possibilities for $f$ adds up to $(q+1)q^s/2$ as claimed.
\end{proof}

\begin{corol}
Assuming standard polynomial arithmetic in $\Fq[t]$, Algorithm 2 requires $O(B^4
q^B) = O(q^{B+\epsilon})$ operations in $\Fq$ as $B \rightarrow \infty$.  The $O$-constant is cubic in $q$ when $B$ is odd and quartic in $q$ when $B$ is even.
\label{algislinanalysisunus}
\end{corol}
\begin{proof}
Algorithm \ref{tabimagbetter} loops exactly over the forms in $\mathcal{F}_s$ for $s \leq B$, with
$\mathcal{F}_s$ as in Lemma \ref{l:forms1}. For each such form $f$, the entire
collection of polynomial computations in Algorithm \ref{tabimagbetter}, including those of Algorithm
\ref{cubformtestversion2}, requires at most $K s^2$ field operations for some constant $K$ that is
independent of $B$ and $q$. This holds because all polynomials under consideration
have degree bounded by~$s$.

The overall complexity of Algorithm \ref{tabimagbetter} is dominated by step 20. The sign and degree
checks in that step, as well as the test for partial reducedness of $H_f = (P, Q,
R)$, require negligible time. When $|P| = |R|$, the $q+1$ partially reduced forms
equivalent to $H_f$ must be computed and compared to $H_f$. Once the reduced
Hessian is found, one needs to check if it has non-trivial automorphisms according
to Theorem \ref{distinctalphas}. If yes, compute the two
corresponding equivalent binary cubic forms with the same Hessian and identity the
lexicographically smallest among these three cubic forms. It follows that the
overall complexity of Algorithm 2 is certainly bounded above by
\[ \sum_{s=1}^B \big ( \#\mathcal{F}_s + 2(q+1) E_s \big ) \cdot K s^2 \ .\]
By Lemma \ref{l:forms2}, $2(q+1) E_s = O(q^s)$, so this contribution is negligible
compared to the error term in the bound on $\#\mathcal{F}_s$ given in Lemma
\ref{l:forms1}. Hence, the asymptotic complexity of Algorithm 2 is bounded above
by
\[ \sum_{s=3}^B \left ( \left(C_s s^2 q^s + O(sq^s)\right) \cdot K s^2 \right )
    \leq  \sum_{s=3}^B C_s q^s \cdot K B^4 + O(B^3 q^B) \ , \]
where $C_s = q^3/32$ if $s$ is odd and $C_s = q^4/32$ is $s$ is even.

Suppose first that $B$ is odd. Then
\[ \sum_{s=3}^B C_s q^s = \sum_{i=1}^{(B-1)/2} \frac{q^3}{32} \, q^{2i+1}
            + \sum_{i=2}^{(B-1)/2} \frac{q^4}{32} \, q^{2i}
    < \frac{q^5}{16(q^2-1)} \, q^B \ . \]
Similarly, if $B$ is even, then
\begin{eqnarray*}
\displaystyle \sum_{s=3}^B C_s q^s
    & = & \displaystyle \sum_{i=1}^{B/2-1} \frac{q^3}{32} \, q^{2i+1}
            \ + \ \sum_{i=2}^{B/2} \frac{q^4}{32} \, q^{2i} \\
    & < & \displaystyle \frac{q^4}{32(q^2-1)} \, q^B + \frac{q^6}{32(q^2-1)} \, q^B
        \ = \ \frac{q^4(q^2+1)}{32(q^2-1)} \, q^B \ .
\end{eqnarray*}
Hence, the overall run time is $O(B^4 q^B)$ where the $O$-constant is as claimed.
\end{proof}

The bounds in Corollary \ref{algislinanalysisunus} appear to be reasonably sharp. For any fixed
$q$, the time required to run Algorithm~\ref{tabimagbetter} on the discriminant degree bound $B+2$
(of the same parity) should be larger by a factor of $q^2$, compared to the time
required when using the bound $B$. If $B$ is odd, then the computation time of
Algorithm \ref{tabimagbetter} using the (even) discriminant degree bound $B+1$ should be larger by a
factor of $(1+B^{-1})^4(q^2+1)/2$, compared to the time required when running the
algorithm with the bound~$B$. Similarly, going from an even bound $B$ to $B+1$
should increase the run time by a factor of $(1+B^{-1})^4 \cdot 2q^2\!/(q^2+1)$.
Our computations times in Tables 1 and 2 largely bear this out.

If we write $q^B = X$, i.e.\ $|D| \leq X$, then the complexity of Algorithm 2 is $O(X^{1 + \epsilon})$ as $X \rightarrow \infty$, which is completely analogous to the run time of Belabas' algorithm [3] for tabulating cubic number fields of absolute discriminant up to $X$.

Corollary \ref{algislinanalysisunus} states that our algorithm in the imaginary and unusual cases should be roughly linear in $X$ if $q$ is small.  To see that this is the case in practice, we plotted the various values of $\log_q (\mathrm{sec})$ versus $\log_q X$ (i.e.\ degree) for $q=5$ and $q=7$, where sec denotes the time (in seconds) taken to tabulate all cubic function fields whose discriminant has absolute value at most $X$ for various values of $\log_q X$.  The line of best fit for the data in the imaginary and unusual Hessian cases is also given in each figure.  As seen in Figures \ref{F5imagplot} and \ref{F7imagplot}, the running times (in seconds) of Algorithm \ref{tabimagbetter} are approximately linear in $X$ for both imaginary and unusual Hessian, as expected.

\begin{figure}
\centering
\includegraphics[width=350pt, height=219pt]{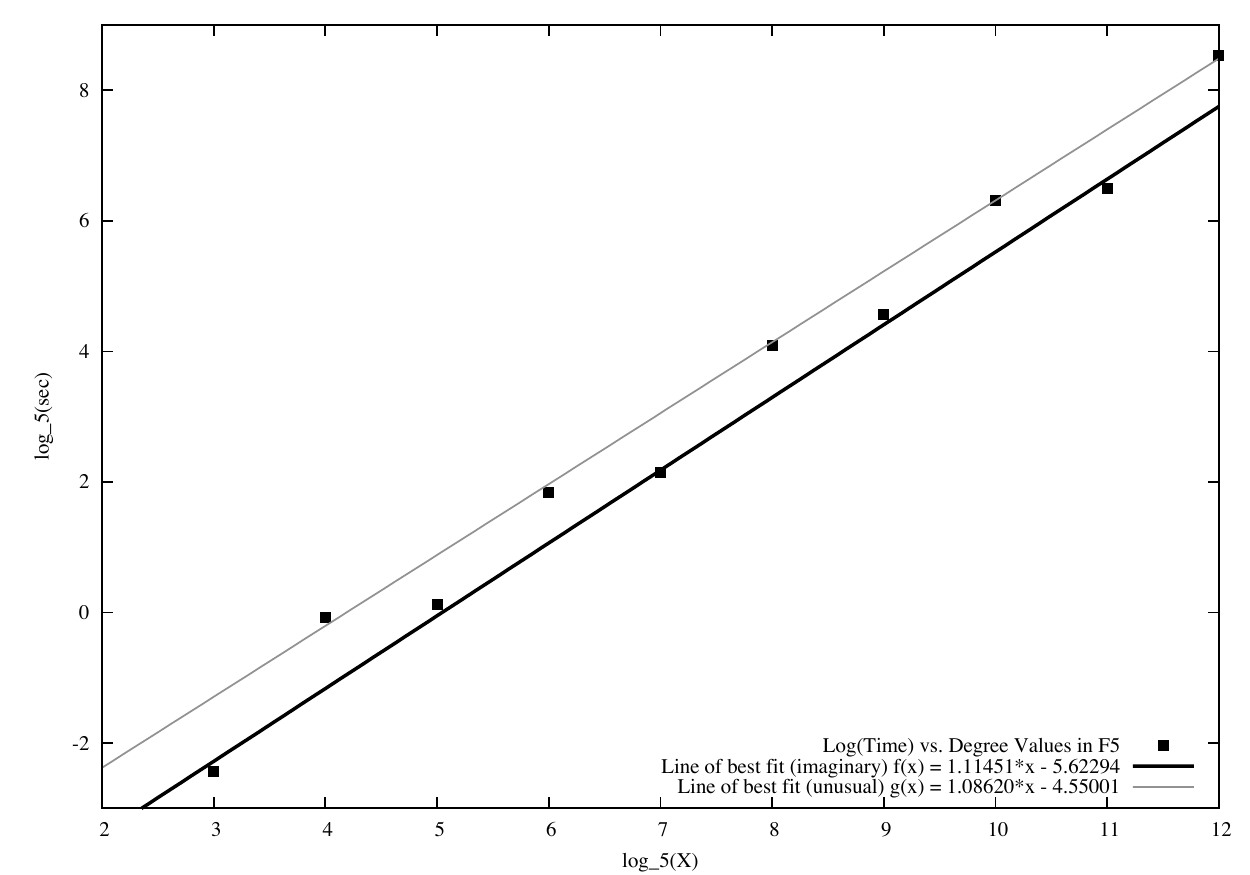}
\caption{ $\finfldq{5}$ Tabulation Timings for Cubic Function Fields:  $\log_5(sec)$ versus degree}
\label{F5imagplot}
\end{figure}

\begin{figure} 
\centering
\includegraphics[width=350pt, height=219pt]{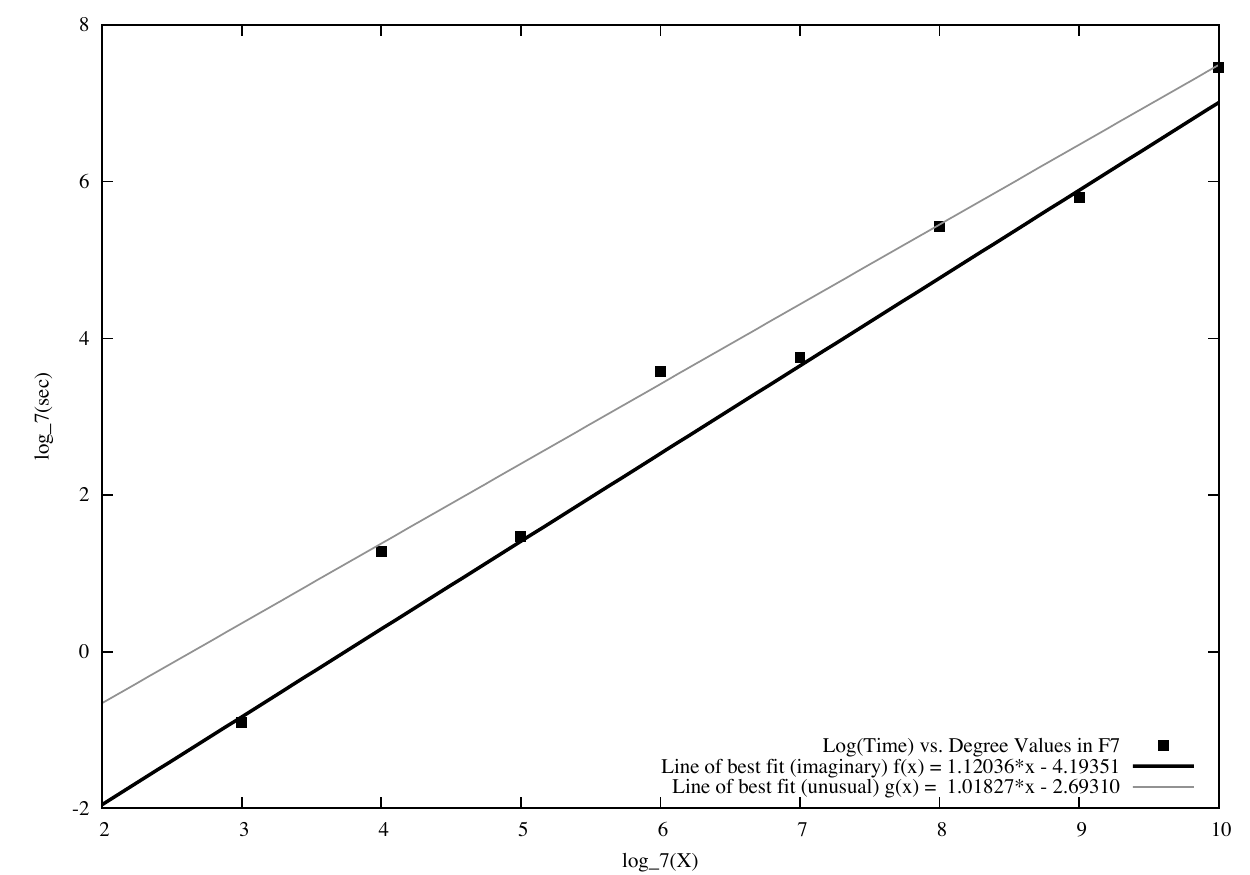}
\caption{ $\finfldq{7}$ Tabulation Timings for Cubic Function Fields:  $\log_7(sec)$ versus degree}
\label{F7imagplot}
\end{figure}

The algorithm presented in this section has some of the same advantages as
Belabas' algorithm \cite{Bela1} over earlier field tabulation algorithms (see Cohen \cite{Coh2}, Chapter 9).  In particular, by Theorem \ref{imaginuirr}, there is no need
to check for irreducibility over $\ratsfuncq{q}$ of binary cubic forms lying in
$\newU$, no need to factor the discriminant, and no need to keep
all fields found so far in memory.    Our algorithm has the
additional advantage that there is no overhead computation needed
for using a sieve to compute numbers that are not square-free,
by the remarks following Algorithm \ref{cubformtestversion2}.

The number of binary cubic forms with Hessian having non-trivial automorphisms was rare.  The percentage of such cubic fields with discriminant degree at most $B$ having non-trivial automorphisms appears to be tending towards zero as $B \rightarrow \infty$.  We conjecture that this rare behavior persists for higher degrees, but analyzing the behavior of such fields and their distribution remains an open problem.

\section{Conclusion and Open Problems}\label{sec:Concl}

The main results of this paper are the development of the reduction theory of binary cubic forms with coefficients in $\polyringq{q}$, and its use in conjunction with the Davenport-Heilbronn theorem to obtain an algorithm for tabulating cubic function fields.   The tabulation algorithm checks $O(q^{B+\epsilon})$ reduced forms in order to tabulate all cubic function fields with imaginary or unusual Hessian whose discriminant satisfies $\deg(D) \leq B$, which is in line with Belabas' result \cite{Bela1} for number fields when $q$ is small.  

The reduction theory developed here is applicable to cubic forms (and hence
function fields) of discriminant $D$ when $-3D$ is imaginary or unusual. It is unclear which suitable quadratic form should be associated to a cubic form of discriminant $D$ when $-3D$ is real. Neither the Mathews \cite{Math1}, Berwick and Mathews \cite{BerMath}, nor the Julia
approach \cite{Crem1} appear to be applicable in general here; even if there are certain cases where they might lead to a unique representative in each equivalence class of cubic forms, it is unclear how to derive upper bounds on the coefficients of such a form, due to the non-Archimedian nature of the absolute value on $\polyringq{q}$.

One possible way to overcome this obstacle is to change the question somewhat. Instead of considering cubic extensions $F/\ratsfuncq{q}$ of discriminant $D$ up to some degree bound, we consider such extensions whose \emph{ramification divisor} (or \emph{different}) $\mathfrak{D}$ has a norm which satisfies the degree bound. Here, $\mathfrak{D}$ incorporates the information on all the ramified places, including the infinite ones, while only the finite places are contained in the discriminant $D$. We have $\deg(\norm{\mathfrak{D}}) = \deg(D) + \epsilon_F$ where $\epsilon_F$  ($0 \leq \epsilon_F \leq 2$) is
given by the ramified infinite places of $F$ and can be computed from the signature at infinity of $F$ (see \cite{cuffqi,LRWWS,Sch1}). Here, one needs to understand the relationship between the ramification divisor of a cubic form and that of a suitable associated quadratic form.   A more detailed exploration of this approach is the subject of future work.

From our tabulation output, it appears that the number of cubic extensions over $\ratsfuncq{q}$ with odd discriminant degree is always divisible by $q(q-1)$.  For imaginary Hessians, the divisibility by $q$ is easily explained: every one of the $q$ translates $t \rightarrow t+u$ with $u \in \finfldq{q}^*$ keeps a form reduced since it does not change any degrees or signs. The resulting form is different unless $a,b,c,d$ are all polynomials in $t^q - t$, which is impossible from the degree bounds in Theorem \ref{acoefbound} for reasonably sized $D$: unless $\deg(D) \geq 4q$, $a$ and $b$ must be constant, which forces $\deg(c) = \deg(d) \geq q$. Then $\deg(D) = 3 \deg(c) \geq 3q$, which is still very large.  For unusual Hessians, the above argument does not apply is as the lexicographical ordering would not be preserved under these translates.   It would be interesting to be able to prove if this type of divisibility phenomenon always occurs, or at the very least, prove specific formulas for fixed even discriminant degree and $q$ values.  We discuss this in an upcoming paper \cite{ANTS9Pie}.

An explicit comparison to the Datskovsky-Wright asymptotics on cubic function fields \cite{DatsWright} was not completed here, since we did not consider the case where $-3D$ is a real discriminant.  Furthermore, the asymptotics are not given for each possible signature for $-3D$.  Other asymptotics on function fields of arbitrary degree which take into account the Galois group include Ellenberg and Venkatesh \cite{EllenVenk1}.  Density results for number fields can be found in \cite{Bhar2,EllenVenk2}, among others.

Constructing tables of number fields has been done for cubic, quartic
and other higher degree extensions (see Cohen \cite{Coh2}).  To the knowledge of the authors, the
problem of tabulation of function fields has not been widely explored.
The generalization of existing algorithms used for tabulating number fields to the function field setting is also the subject of future work.   
 
Belabas modified his tabulation algorithm to compute $3$-ranks of quadratic number fields \cite{Bela2}.  This has also been generalized to quadratic function fields in \cite{Pieter2}, and is the subject of a future paper \cite{ANTS9Pie}.\\

\noindent \emph{Acknowledgements} \hspace{3mm}  The authors would like to thank an anonymous referee for a thorough review of this work and for a number of very helpful suggestions that led to significant improvements to this paper, specifically for the complexity analysis of our algorithm.

\bibliographystyle{amsplain}

\begin{thebibliography}{75}
\bibitem{Artin1} E. Artin, Quadratische K\"{o}rper im Gebiete
    der h\"{o}heren Kongruenzen I, \emph{Math. Zeitschrift} \textbf{19}
  (1924), 153--206.
\bibitem{BachShall} E. Bach and J. Shallit, \emph{Algorithmic {N}umber {T}heory. {V}ol. 1:  Efficient Algorithms}, Foundations of Computing Series, MIT Press, Cambridge, MA, 1996.
\bibitem{Bela1} K. Belabas, A fast algorithm to compute cubic
    fields, \emph{Math. Comp.} \textbf{66} (1997), no. 219, 1213--1237.
\bibitem{Bela2} K. Belabas, On quadratic fields with large
    $3$-rank, \emph{Math. Comp.} \textbf{73} (2004), no. 248, 2061--2074.
\bibitem{BerMath} W.E.H. Berwick and G.B. Mathews, On the reduction of arithmetical binary cubics which have negative discriminant, \emph{Proc. of the London Math. Soc.}, \textbf{10} (1912), 48--53.
\bibitem{Bhar2} M. Bhargava, The density of discriminants of quartic rings and fields, \emph{Annals of Math. Second Series} \textbf{162} (2005), no. 2, 1031--1063.
\bibitem{BuchVoll} J. Buchmann and U. Vollmer, \emph{Binary Quadratic Forms:  An Algorithmic Approach}, Algorithms and Computation in Mathematics \textbf{20}. Springer, Berlin, 2007.
\bibitem{Buell1}  D.A. Buell, \emph{Binary Quadratic Forms -- Classical Theory
                         and Modern Computations}, Springer-Verlag, New York, 1989.
\bibitem{Casse1} R. Casse, \emph{Projective {G}eometry: an {I}ntroduction}, Oxford University Press, Oxford, 2006.
\bibitem{Coh2} H. Cohen, \emph{Advanced Topics in Computational Number
    Theory}, Springer-Verlag, New York, 2000.
\bibitem{CranPom} R. Crandall and C. Pomerance, \emph{Prime {N}umbers:  {A} {C}omputational {P}erspective}, First Edition, Springer, New York, 2001.
\bibitem{Crem1} J.E. Cremona, Reduction of binary cubic and
    quartic forms, \emph{LMS J. Comput. Math.} \textbf{2}
   (1999), 62--92.
\bibitem{DatsWright} B. Datskovsky and D.J. Wright, Density of
    discriminants of cubic extensions, \emph{J. Reine
    Angew. Math.} \textbf{386} (1988), 116--138.
\bibitem{HeDav1} H. Davenport and H. Heilbronn, On the density
    of discriminants of cubic fields I, \emph{Bull. London
    Math. Soc.} \textbf{1} (1969), 345--348.
\bibitem{HeDav2} H. Davenport and H. Heilbronn, On the density
    of discriminants of cubic fields II, \emph{Proc. Royal Soc. London
    A} \textbf{322} (1971), 405--420.
\bibitem{EllenVenk1} J.S. Ellenberg and A. Venkatesh, Counting extensions of function fields with bounded discriminant and specified {G}alois group, In  \emph{Geometric {M}ethods in {A}lgebra and {N}umber {T}heory}, \emph{Progress in Mathematics} \textbf{235}, 151--168, Birkh\"auser Boston, Boston, MA, 2005.
\bibitem{EllenVenk2} J.S. Ellenberg and A. Venkatesh, The number of extensions of a number field with fixed degree and bounded discriminant, \emph{Annals of Math. Second Series}, \textbf{163} (2006), no. 2, 723--741.
\bibitem{Enge1} A. Enge, How to distinguish hyperelliptic curves in even characteristic, \emph{Public-Key Cryptography and Computational Number Theory}, De Gruyter, Berlin, 2001, 49--58.
\bibitem{Hirsch1} J.W.P. Hirschfeld, \emph{Projective Geometries Over Finite Fields}, Second Edition, Oxford Mathematical Monographs, Oxford University Press, New York, 1998.
\bibitem{cuffqi} M.J. Jacobson, Jr., Y. Lee, R. Scheidler and H. Williams, Construction of all cubic function fields of a given square-free discriminant, preprint.
\bibitem{LRWWS} E. Landquist, P. Rozenhart, R. Scheidler, J. Webster and Q. Wu, An explicit treatment of cubic function fields with applications, \textit{Canadian J. Math.}, \textbf{62} (2010), no. 4, 787--807, available online at \texttt{http://www.cms.math.ca/10.4153/CJM-2010-032-0}.
\bibitem{LidlNied} R. Lidl and H. Niederreiter, \emph{Introduction to {F}inite {F}ields and {T}heir {A}pplications}, Cambridge University Press, Cambridge, 1994.
\bibitem{Math1} G. B. Mathews, On the reduction and classification of binary cubics
                  which have a negative discriminant, \emph{Proc. of the London Math. Soc.}, \textbf{10} (1912), 128--138.
\bibitem{Rosen} M. Rosen, \emph{Number Theory in Function Fields},
  Springer-Verlag, New York, 2002.
\bibitem{Pieter2} P. Rozenhart, \emph{Fast Tabulation of Cubic
    Function Fields}, Ph.D.\ Thesis, University of Calgary, 2009.
\bibitem{ANTSPie} P. Rozenhart and R. Scheidler, Tabulation of cubic function fields with imaginary and unusual {H}essian, \emph{Proc.\ Eighth Algorithmic Number Theory Symposium ANTS-VIII}, In \emph{Lecture Notes in Computer Science \textbf{5011}}, 357--370, Springer, 2008.
\bibitem{ANTS9Pie} P. Rozenhart, M.J. Jacobson, Jr., and R. Scheidler, Computing quadratic function fields with high $3$-rank via cubic field tabulation, preprint.
\bibitem{Sch1} R. Scheidler, Algorithmic aspects of cubic function fields, \emph{Proc.\ Sixth Algorithmic Number Theory Symposium ANTS-VI}, In \emph{Lecture Notes in Computer Science \textbf{3076}}, 395--410, Springer, 2004.
\bibitem{Shoup1} V. Shoup, \emph{NTL:  A Library for Doing Number
    Theory}, Software, 2001, see http://www.shoup.net/ntl.
\bibitem{Sticht} H. Stichtenoth, \emph{Algebraic Function Fields and
    Codes}, Second Edition, Springer-Verlag, New York, 2009.
\bibitem{Taniguchi} T. Taniguchi, "Distributions of discriminants of
  cubic algebras", Preprint, Available from
  \texttt{http://arxiv.org/abs/math.NT/0606109} (2006). 
\end{thebibliography}

\end{document}